\documentclass[AMA,Times1COL]{WileyNJDv5} 

\usepackage{amsmath}
\usepackage{xcolor}
\usepackage{mathtools}
\usepackage{subcaption}

\articletype{Research Article}%

\received{Date Month Year}
\revised{Date Month Year}
\accepted{Date Month Year}
\journal{Journal}
\volume{00}
\copyyear{2023}
\startpage{1}

\begin{document}

\title{Forward-Backward Dynamic Programming for LQG Dynamic Games with Partial and Asymmetric Information}

\author[1]{Yuxiang Guan}

\author[2]{Iman Shames}

\author[1]{Tyler Summers}

\authormark{GUAN \textsc{et al.}}
\titlemark{Forward-Backward Dynamic Programming for LQG Dynamic Games with Partial and Asymmetric Information}

\address[1]{\orgdiv{Department of Mechanical Engineering}, \orgname{The University of Texas at Dallas}, \orgaddress{\state{TX, 75080-3021}, \country{USA}}}


\address[2]{\orgdiv{Department of Electrical and Electronic Engineering}, \orgname{The University of Melbourne}, \orgaddress{\state{VIC, 3010}, \country{Australia}}}

\corres{Yuxiang Guan \email{yuxiang.guan@utdallas.edu}}



\fundingInfo{United States Air Force Office of Scientific Research under Grants FA9550-23-1-0424 and FA2386-24-1-4014, and by the National Science Foundation under Grant ECCS-2047040.}

\abstract[Abstract]{
We formulate and study a class of two-player zero-sum stochastic dynamic games with partial and asymmetric information. Information asymmetry introduces fundamental challenges involving \emph{belief representation} and \emph{theory of mind} issues, where agents must impute belief states and estimates of other agents to inform their own strategy. To avoid an infinite regress of higher-order beliefs amongst agents and obtain computationally implementable results, we focus on a linear quadratic Gaussian (LQG) model and consider strategies with limited internal state dimension. We present a novel iterative forward-backward algorithm to jointly compute belief states and equilibrium strategies and value functions for a finite-horizon problem. We also present a value iteration-like algorithm to jointly compute stationary belief states and equilibrium strategies for an average-cost infinite-horizon problem. An open-source implementation of the algorithms is provided, and we demonstrate the effectiveness of the proposed algorithms in numerical experiments.
}



\maketitle

\renewcommand\thefootnote{}

\renewcommand\thefootnote{\fnsymbol{footnote}}
\setcounter{footnote}{1}

\section{Introduction}\label{introduction}

Dynamic game theory provides a framework for analyzing and designing feedback strategies for decision-making agents that strategically interact in dynamic and uncertain environments \cite{bacsar1998dynamic,yuksel2023stochastic}. A prominent application of this framework is in robust control, where disturbances are modeled as adversarial players maximizing performance degradation \cite{bernhard1996min}. This game-theoretic perspective unifies $H_\infty$ control with min-max differential games and extends to stochastic settings through risk-sensitive formulations \cite{campi1996nonlinear} and multi-player differential games \cite{li2008stochastic}. Much of this theory assumes that all model parameters and the underlying system state are common knowledge for all agents, in which case dynamic programming techniques can be used to compute equilibrium feedback strategies. However, in many practical settings, agents must make decisions based on partial and asymmetric information about the models and state, fundamentally altering the nature of optimal strategies and the structure of equilibrium solutions.

The information asymmetry introduces fundamental challenges involving \textit{belief representation} and \textit{theory of mind}, where players must impute the belief states and estimates of other players to inform their strategy. This leads to an infinite regress of higher-order beliefs, which renders even representing strategies and value functions difficult. A further challenge is that the information asymmetry leads to motivation for signaling, obfuscation, and deception, in which the players utilize their actions to affect the information of other players to indirectly optimize their objective. These challenges have been recognized in static game theory since the 1960s, including in the seminal work of Harsanyi \cite{harsanyi1967games}. In the context of dynamic games, it means that control strategies may be infinite dimensional, or equivalently, each agent must store and use its entire history of available information at each time to make optimal decisions.

Dynamic games with partial and asymmetric information have also been studied since the 1960s \cite{ho1966optimal,behn1968class,rhodes1969differential,rhodes1969stochastic,ho1970differential,tamer1973multistage,basar1978decentralized,ba1978two}, though far less extensively than dynamic games with perfect state information. Recent work has addressed many challenges and made substantial progress on various aspects \cite{nayyar2013common,gupta2014common,bacsar2014stochastic,ouyang2016dynamic,sinha2016structured,gupta2016dynamic,pachter2017lqg,vasal2018systematic,heydaribeni2025structured,saritacs2020dynamic,vasal2021signaling,kartik2021upper,tang2023dynamic,arslan2023subjective,hambly2023linear,ouyang2024approach,arslan2024subjective,guan2025best}, which we review in detail in the next subsection. Much of the existing literature focuses on constructing LQG dynamic games or partially observable Markov decision process (POMDP) with special information structures and strategy spaces that enables the separation of estimation and control, followed by sequential decomposition or other simplifications. However, signaling incentives arising from the information asymmetry generally prevent the complete separation of estimation and control and significantly complicate the computation of equilibrium strategies.

In such games, a widely used equilibrium concept is perfect Bayesian equilibrium (PBE) \cite{fudenberg1991game}, a signaling equilibrium where players' beliefs depend on their strategies. Several decomposition methods have been developed to compute PBE and its variants \cite{sinha2016structured,saritacs2020dynamic,vasal2021signaling,tang2023dynamic,ouyang2024approach}, though these works employ modeling frameworks and approaches that differ from ours. In general, there does not exist a universal algorithm that decouples the interdependence of strategies and beliefs over time in calculating PBE. Best response dynamics have been explored in solving a two-player zero-sum LQG dynamic game with partial and asymmetric information \cite{guan2025best}, motivated by the observation that players' higher-order belief states can be approximated by low-order belief states with bounded error, allowing feedback strategies with limited internal state dimension to closely approximate PBE.

We formulate and analyze a class of two-player zero-sum LQG dynamic games with partial and asymmetric information. Under certain assumptions, we develop an explicit, computationally tractable forward-backward algorithm for computing an approximate perfect Bayesian equilibrium (APBE). Unlike existing work that relies on certain common information structures to separate estimation and control, our approach directly handles the general lack of such separation. Our main contributions are:
\begin{enumerate}
    \item We derive coupled forward and backward recursions for propagating belief states (Proposition \ref{prop1:partialinfofilters}) and computing cost-to-go functions and state estimate feedback strategies (Proposition \ref{DPalgorithm});
    \item We develop a novel iterative forward-backward algorithm to jointly compute belief states and equilibrium strategies for two-player zero-sum finite-horizon LQG dynamic games with partial and asymmetric information (Algorithm \ref{algorithm:forward-backward});
    \item We present a value iteration-like algorithm to jointly compute stationary belief states and equilibrium strategies for an average-cost infinite-horizon version of the problem (Algorithm \ref{algorithm:forward-backward-VI});
    \item To facilitate further progress in the area, we have released an open-source implementation of these algorithms, available at \url{https://gitlab.com/scratch7473433/partialinfodynamicprogramming};
    \item We show numerical experiments that demonstrate the effectiveness of the proposed algorithms and the obtained equilibrium strategies.
\end{enumerate}

\subsection*{Related Work}
Some of the key challenges in dynamic games with incomplete information were first identified and discussed in Chapter 12 of Isaacs' classic book on differential games \cite{isaacs1965differential}. Classes of pursuit-evasion differential games in which one of the players has limited information were studied by Ho \cite{ho1966optimal}, Behn and Ho \cite{behn1968class}, and Rhodes and Luenberger \cite{rhodes1969differential,rhodes1969stochastic}. Ho \cite{ho1966optimal} considered games where only one player's action affects the state and attempts to confuse the other player's attempt to estimate the state. Differential games in which one or both players has either perfect state information or uses an open-loop strategy have been examined by Behn and Ho \cite{behn1968class}, Rhodes and Luenberger \cite{rhodes1969differential}, and Ba\c{s}ar and Mintz \cite{tamer1973multistage}. The special structure in these problems allows classical separation and certainty-equivalence properties from single-player optimal control. Differential games in which both players have partial and asymmetric information were studied by Rhodes and Luenberger \cite{rhodes1969stochastic} and Willman \cite{willman1971some}. 
These early works already acknowledged the possibility that the equilibrium strategies could be infinite-dimensional due to a regress into higher-order beliefs, requiring strategies to be based on the entire history of available information.   

More recent work has sought to identify and study broader classes of information structures that still enable separation of estimation and control and then computation of strategies via sequential decomposition and dynamic programming techniques. Nayyar et al. \cite{nayyar2013common} and Gupta et al. \cite{gupta2014common, gupta2016dynamic} study a class of dynamic games where agent beliefs are based on common information and independent of the strategies employed. This structure eliminates signaling incentives, allowing a class of equilibria called Common Information Based Markov Perfect Equilibria to be characterized and computed via dynamic programming. Similarly, Pachter \cite{pachter2017lqg} examines a control-sharing information pattern that renders the game amenable to dynamic programming solutions. Tang et al. \cite{tang2023dynamic} investigate another information pattern with delayed observation sharing amongst teams, developing a backward induction solution approach.

Ouyang et al. \cite{ouyang2016dynamic} also employ a sequential decomposition approach to compute PBE in a class of partial information dynamic games with signaling. Vasal and Anastasopoulos \cite{vasal2018systematic} compute a special class of structured PBE using a forward-backward decomposition and a notion of player type to represent private information of players; the related work \cite{vasal2021signaling} by the same authors considers a linear-quadratic setting. Kartik and Nayyar \cite{kartik2021upper} compute upper and lower values for dynamic games with finite state spaces under several potentially asymmetric information structures and signaling, using a dynamic programming approach. Ouyang et al. \cite{ouyang2024approach} develop a sequential backward recursion for computing Bayesian Nash Equilibrium strategies based on compressed private and common information. Arslan and Yüksel \cite{arslan2023subjective} study a notion of subjective equilibrium under beliefs of exogenous uncertainty.

The work discussed above still has a strong focus on approaches that exploit the separation of estimation and control and can piece together equilibrium strategies with dynamic programming-type backward recursions. Further, to our best knowledge, none of these works have provided open-source, extensible computational implementations. In contrast, our combined forward-backward algorithm is not necessarily limited to problems with complete separation of estimation and control, and we provide an open-source implementation to facilitate further research.





Perspectives on games with asymmetric information have yielded groundbreaking work in other fields. Identification of crucial and surprising effects of information asymmetry in game theory has led to Nobel prize-winning work in information economics \cite{harsanyi1967games, akerlof1970market, aumann1976agreeing, spence1978job, stiglitz1981credit}. The seminal work of Harsanyi \cite{harsanyi1967games} developed a framework for games with imperfect information, defined a now widely studied class of Bayesian equilibria, and introduced the notion of ``type'' as an approach to avoid an infinite regress of higher-order beliefs.

The remainder of this work is organized as follows. Section \ref{sec:finite-horizon-LQG-game} defines a two-player zero-sum finite-horizon LQG dynamic game with partial and asymmetric information and develops a forward-backward algorithm to jointly compute belief states and equilibrium strategies of an APBE. Section \ref{Section:InfHorizon} extends this framework to the infinite-horizon setting, presenting a value iteration-like algorithm to jointly compute stationary belief states and equilibrium strategies of an APBE for the average-cost formulation. Section \ref{sec:numerical_experiments} demonstrates the effectiveness of the proposed algorithms through numerical experiments across various scenarios. Finally, Section \ref{sec:extensions} discusses several extensions and variations of the modeling framework and methodology.

\section{Finite Horizon LQG Dynamic Games with Partial and Asymmetric Information}\label{sec:finite-horizon-LQG-game}
\subsection{Problem Formulation}
\textbf{System Dynamics and Measurements.} We consider a two-player zero-sum finite-horizon LQG dynamic game with partial and asymmetric information. 
The system dynamics are linear, given by
\begin{equation}\label{eqn:sys_dynamics}
    x_{t+1} = A_t x_t + B_t^1 u_t^1 + B_t^2 u_t^2 + w_t,
\end{equation}
where $x_t \in \mathbf{R}^n$ is the system state, $u_t^1 \in \mathbf{R}^{m_1}$ is the control input for player $1$, $u_t^2 \in \mathbf{R}^{m_2}$ is the control input for player $2$, $w_t \sim \mathcal{N}(0, W_t)$ is a Gaussian random vector with zero mean and covariance matrix $W_t$. The initial state $x_0$ may be either fixed or a Gaussian random vector with mean $\bar x_0$ and covariance matrix $X_0$, i.e., $x_0 \sim \mathcal{N}(\bar x_0, X_0)$.

To model settings with partial/incomplete and asymmetric information, the players do not have exact knowledge of the state and actions of other players. Instead, each player receives at time $t$ a signal $y^i_t$ through noisy sensor measurements, which is modeled by the linear output equations
\begin{equation}\label{eqn:output}
\begin{aligned}
    y_t^1 &= C_t^1 x_t + v_t^1, \\
    y_t^2 &= C_t^2 x_t + v_t^2,
\end{aligned}
\end{equation}
where $y_t^1 \in \mathbf{R}^{p_1}$, $y_t^2 \in \mathbf{R}^{p_2}$, $v_t^1 \sim \mathcal{N}(0, V_t^1)$ and $v_t^2 \sim \mathcal{N}(0, V_t^2)$ are a Gaussian random vectors with zero mean and covariance matrices $V_t^1$ and $V_t^2$, respectively.

\textbf{Information Structure and Strategy Space.}
Since the players do not know the exact value of the state, they must in general make decisions based on a history of all available information. Accordingly, we define the information states
\begin{equation}\label{eqn:info}
\begin{aligned}
    I^1_t &= \left(y_0^1, y_1^1,\ldots,y_t^1, u_0^1, u_1^1,\ldots,u_{t-1}^1\right), \\
    I^2_t &= \left(y_0^2, y_1^2,\ldots,y_t^2, u_0^2, u_1^2,\ldots,u_{t-1}^2\right).
\end{aligned}
\end{equation}
In general, players seek strategies defined as functions of their available information sets. Player $1$ adopts the strategy $\pi \coloneqq (\pi_0, \pi_1, \ldots, \pi_{T-1})$, where $u_t^1 = \pi_t(I_t^1)$, and player $2$ adopts $\mu \coloneqq (\mu_0, \mu_1, \ldots, \mu_{T-1})$, where $u_t^2 = \mu_t(I_t^2)$. But to limit the complexity of the strategies and avoid an infinite regress of higher-order beliefs, we restrict attention to strategies that depend on state estimates with limited dimension rather than the full information states. Specifically, we consider strategies of the form $\pi = (\pi_0, \pi_1, \ldots, \pi_{T-1})$ with $u_t^1 = \pi_t(z_t^1)$ and $\mu = (\mu_0, \mu_1, \ldots, \mu_{T-1})$ with $u_t^2 = \mu_t(z_t^2)$, where $z_t^1, z_t^2 \in \mathbf{R}^n$ are state estimates processed from the information states $I_t^1$ and $I_t^2$ using the following linear filters
\begin{equation} \label{onestepfilters}
\begin{aligned}
    z^1_{t+1} &= A_t^1 z_t^1 + \bar B_t^1 u_t^1 + \bar L_t^1 \left(y_t^1 - C_t^1 z_t^1\right), \\
    z^2_{t+1} &= A_t^2 z_t^2 + \bar B_t^2 u_t^2 + \bar L_t^2 \left(y_t^2 - C_t^2 z_t^2\right).
\end{aligned}
\end{equation}
The filter matrices $A_t^1$, $A_t^2$, $\bar B_t^1$, $\bar B_t^2$, $\bar L_t^1$, and $\bar L_t^2$ will be characterized in Proposition \ref{prop1:partialinfofilters}. We further restrict attention to linear feedback strategies of the form
    \begin{equation} \label{linearestimatefeedback}
    \begin{aligned}
        u_t^1 &= K_t^1 z_t^1, \\
        u_t^2 &= K_t^2 z_t^2,
        \end{aligned}
    \end{equation}
with feedback gain matrices $ K_t^1 \in \mathbf{R}^{m_1 \times n}$ and $ K_t^2 \in \mathbf{R}^{m_2 \times n}$ for $t=0, \ldots, T-1$.
\begin{remark}
    Without special information structures such as partially nested \cite{ho1972team} and stochastically nested \cite{yuksel2009stochastic} configurations, linear control strategies are not optimal in general for the LQG dynamic games with asymmetric information \cite{witsenhausen1968counterexample}. However, linear strategies provide a natural and computationally tractable class of policies for LQG problems while still capturing essential strategic interactions.
\end{remark}
\begin{remark}
The classical output feedback $H_\infty$ control problem can be viewed as a special case of the framework developed here. The key distinction lies in the information structure: the classical setting assumes symmetric information, where both players share the same output observation $y_t$. Our formulation relaxes this restriction, allowing each player to hold a private and distinct information set $I_t^i$, which gives rise to an infinite regress of higher-order beliefs where each player must reason about the other's beliefs, about the other's beliefs about their beliefs, to infinity, rendering the problem significantly more complex than its symmetric counterpart.
\end{remark}

In contrast to the single-player LQG problem, the state estimate for each player cannot be made independent of the strategy of the opposing player, since their strategy is based on private information. This reflects a signaling phenomenon, where each player has an incentive to use its strategy to affect the state estimate of the opposing player. We will discuss in the next subsection how our approach handles this feature of the problem. 

\textbf{Objectives.}
For a fixed feedback strategy pair $(\mathbf{K}^1, \mathbf{K}^2)$, where $\mathbf{K}^1 \coloneqq (K^1_0, K^1_1, \ldots, K^1_{T-1})$ and $\mathbf{K}^2 \coloneqq (K^2_0, K^2_1, \ldots, K^2_{T-1})$, the objective function is
\begin{equation} \label{totalcostLQR}
\begin{aligned}
    J\left(\mathbf{K}^1, \mathbf{K}^2\right)  = \mathbf{E}_{x_0, w_t, v_t^1, v_t^2} \left[ \sum_{t=0}^{T-1} x_t^\top Q_t x_t + z_t^{1\top} K_t^{1\top} R_t K_t^1 z_t^1 + z_t^{2\top} K_t^{2\top} S_t K_t^2 z_t^2 
 + x_T^\top Q_T x_T \right],
\end{aligned}
\end{equation}
where $Q_t \succeq 0$, $R_t \succ 0$, and $S_t \prec 0$, $T$ is the time horizon, and the expectation is taken with respect to the initial state and disturbance and measurement noise sequences. For a zero-sum game, player $1$ seeks a strategy to minimize $J$ while player $2$ seeks a strategy to maximize $J$. The penalty $S_t$ on the maximizing player's input must be sufficiently large to prevent the upper value of the game from being unbounded, regardless of the minimizing player's strategy. When $\lambda_{\max}(S_t) \rightarrow -\infty \ \forall t$, the maximizing player's strategy approaches the zero strategy, and we recover the classical separated single-player LQG solution for the minimizer. 

\begin{assumption}[Common Knowledge]\label{asmp:common_knowledge}
    The system dynamics, all measurement models, and all cost function parameters are assumed to be \emph{common knowledge} of both players. The common knowledge model is denoted by $\mathcal{M} = \{ A_t, B_t^1, B_t^2, W_t, \bar x_0, X_0, C_t^1, C_t^2, V_t^1, V_t^2, Q_t, R_t, S_t \}_{t=0}^T.$
\end{assumption}


\begin{assumption}[Bounded Depth of Reasoning and Self-projecting Beliefs]\label{asmp:belief_projection}
    When determining filter parameters and feedback control strategies, we restrict each player's depth of reasoning to first-order belief. Specifically, we assume that: (1) each player $i$ forms beliefs about player $j$'s belief regarding the state, denoted $\mathbf{E}[z_t^j \mid z_t^i]$, and (2) each player projects their own zeroth-order belief $z_t^i$ (player $i$'s belief about state $x_t$) onto their first-order belief about the other player such that $\mathbf{E}[z_t^j \mid z_t^i]$ = $z_t^i$.
\end{assumption}

\begin{remark}
    Assumption \ref{asmp:common_knowledge} is standard in game-theoretic settings and enables each player to independently reason about their own best responses and those of other players when determining optimal strategies. 
\end{remark}
\begin{remark}
    Assumption \ref{asmp:belief_projection}, which bounds the depth of reasoning to first-order beliefs with self-projection, is the most restrictive of our assumptions. While this may limit players' rationality from certain perspectives, it serves three important purposes: (1) it provides computational tractability by avoiding the infinite regress of higher-order beliefs, (2) it preserves key challenges inherent in these games, most notably the lack of complete separation between estimation and control, and (3) it is supported by numerical observations in LQ games demonstrating that higher-order belief states can be closely approximated by low-order belief states with bounded error \cite{guan2025best}. Similar bounded rationality assumptions have been successfully employed in existing work \cite{schwarting2021stochastic}. 
\end{remark}
In Section \ref{sec:numerical_experiments}, we quantify the resulting optimality gap and approximation error under various scenarios, demonstrating that these assumptions yield effective approximate equilibria in practice.

Under Assumptions \ref{asmp:common_knowledge} and \ref{asmp:belief_projection}, the goal of this work is to obtain an APBE strategy pair $(\mathbf{K}^{1*}, \mathbf{K}^{2*})$ where 
\begin{equation}
    J\left(\mathbf{K}^{1*}, \mathbf{K}^{2}\right) \leq  J\left(\mathbf{K}^{1*}, \mathbf{K}^{2*}\right) \leq  J\left(\mathbf{K}^1, \mathbf{K}^{2*}\right), \quad \forall \ \mathbf{K}^1, \mathbf{K}^2,
\end{equation}
i.e., each player performs worse by deviating from their equilibrium strategy.




In the following subsections, we develop an iterative forward-backward algorithm to compute such equilibrium strategies. The approach consists of a (forward) propagation of belief states using the state estimates \eqref{onestepfilters} for a given strategy pair and a (backward) dynamic programming recursion to compute equilibrium feedback strategies and value functions for given belief states. Due to the partial and asymmetric information and the signaling phenomenon, the forward and backward recursions for estimation and control are coupled and cannot be completely separated in general: the belief propagation depends on the strategies, and the strategies depend on how the beliefs are propagated. Solving the problem requires \emph{jointly} finding filter parameters and control strategies that satisfy both recursions simultaneously. This can be viewed as solving a (discrete-time) two-point boundary value problem. The proposed forward-backward algorithm is reminiscent of block coordinate descent in optimization.

\subsection{Forward State Estimation with Partial and Asymmetric Information} \label{partialinfofilter}
We start with the state estimation problem to characterize the equilibrium filter parameters when the feedback strategies are fixed. The filtering problem involves a forward recursion to propagate belief states as measurements are received. In the LQG setting, the means and a joint covariance matrix are propagated. We will see that the estimation and control problems are coupled and cannot be separated; the equilibrium filters require the feedback control gains and the equilibrium strategies may require the filtering gains.

Each player processes available information using filters of the form \eqref{onestepfilters}. The estimation errors are defined by
\begin{equation}
\begin{aligned}
    e_t^1 = x_t - z_t^1, \quad  e_t^2 = x_t - z_t^2.
    \end{aligned}
\end{equation}
In the single-player state estimation problem, the estimation error can be made independent of the input signal (by selecting $\bar B_t^1 = B_t^1$, so that the input term cancels in the estimation error dynamics). However, in our information pattern, player $i$ does not know player $j$'s input, so their estimation error \emph{cannot be made independent of the opposing player's input}. The control strategies affect the estimation error dynamics, leading to a signaling incentive. This means that the state estimation problem is ill-posed without a specification of the control strategies. Therefore, in this section we assume that the feedback gain matrices $\{ K_t^1 \}$ and $\{ K_t^2\}$ for $t=0,\ldots, T-1$ for the linear strategies \eqref{linearestimatefeedback} are fixed and common knowledge.

Our goal in this section is to find filter parameters $(A_t^1, \bar B_t^1, \bar L_t^1)$ and $(A_t^2, \bar B_t^2, \bar L_t^2)$ for $t=0,\ldots, T-1$ to produce unbiased state estimates that optimize (player $1$ minimizes, player $2$ maximizes) the expected total weighted estimation error
\begin{equation} \label{estimationerror}
\mathbf{E}_{w_t ,  v_t^1, v_t^2  } \sum_{t=0}^T  \begin{bmatrix}
     e_t^1 \\ e_t^2 
 \end{bmatrix}^\top
 \begin{bmatrix}
     \Gamma_t^{11} & \Gamma_t^{12}  \\
      \Gamma_t^{12\top} & \Gamma_t^{22} 
 \end{bmatrix}
  \begin{bmatrix}
     e_t^1 \\ e_t^2
 \end{bmatrix},
\end{equation}
where the weight matrix diagonal blocks satisfy $\Gamma_t^{11} \succ 0$ and $\Gamma_t^{22} \prec 0$.

\begin{remark}
    Unlike the standard Kalman filter where the filter design is independent of how we assign weights to estimation errors, in the two-player zero-sum LQG dynamic game the weight matrix $\Gamma_t$ does affect the equilibrium filter gains. When $\Gamma_t^{12} = 0$, the equilibrium filter gains are independent of the diagonal blocks $\Gamma_t^{11}$ and $\Gamma_t^{22}$. However, the feedback control strategies induce coupling in the estimation error covariance. If the feedback control gains are zero, the error dynamics decouple, and the joint filter reduces to two independent standard Kalman filters. The weighted estimation error also captures players' incentives to use their strategies to affect the information of other players to indirectly optimize their objectives. 
\end{remark}

The joint error covariance matrix is denoted by and partitioned as
$$\Sigma_t := \mathbf{E}\begin{bmatrix} e_t^1 \\ e_t^2 \end{bmatrix} \begin{bmatrix} e_t^1 \\ e_t^2 \end{bmatrix}^\top = \begin{bmatrix} \Sigma_t^{11} & \Sigma_t^{12} \\ \Sigma_t^{12\top} & \Sigma_t^{22}\end{bmatrix}.$$

It is common to split the belief propagation into a priori and a posteriori steps, where an a priori step updates the state estimate and covariance using the dynamic model before a measurement is received, and an a posteriori step updates those quantities upon receiving output measurements. We use $z_t^{i-}$ and $z_t^{i+}$ to denote the a priori and a posteriori estimates, respectively, and $\Sigma_t^{-}$ and $\Sigma_t^{+}$ to denote the a priori and a posteriori joint estimation error covariances, respectively.

\begin{proposition} \label{prop1:partialinfofilters}
Under Assumptions \ref{asmp:common_knowledge} and \ref{asmp:belief_projection}, the following forward propagation of state estimates for each player optimizes the estimation error in \eqref{estimationerror}:
\begin{itemize}
\item \textbf{Initialize:}
\begin{equation} \label{initialize}
\begin{aligned}
z_0^{1+} &= z_0^{2+} = \mathbf{E} x_0 = \bar x_0, \\
\Sigma_0^+ &=   \begin{bmatrix} X_0  & 0 \\ 0 & X_0 \end{bmatrix}.
\end{aligned}
\end{equation}

\item \textbf{A priori update:} For $t=1,\ldots, T-1$
\begin{equation} \label{apriori}
\begin{aligned}
z_{t+1}^{1-} &= \left(A_t + B_t^2 K_t^2\right) z_t^{1+} + B_t^1 u_t^1, \\
z_{t+1}^{2-} &= \left(A_t + B_t^1 K_t^1\right) z_t^{2+} + B_t^2 u_t^2, \\
\Sigma_{t+1}^- &=  \bar A_t \Sigma_t^+ \bar A_t^\top + \bar W_t, 
\end{aligned}
\end{equation}
\begin{equation} \label{apriorimatrices}
\begin{aligned}
\bar A_0 =  \begin{bmatrix} A_0  & 0 \\ 0 & A_0 \end{bmatrix}, \quad 
\bar A_t =  \begin{bmatrix} A_t + B_t^2 K_t^2  & - B_t^2 K_t^2 \\ -B_t^1 K_t^1 & A_t + B_t^1 K_t^1 \end{bmatrix}, \quad \bar W_t =  \begin{bmatrix} W_t  & 0 \\ 0 & W_t \end{bmatrix}, \ \forall t.
\end{aligned}
\end{equation}

\item \textbf{A posteriori update:} For $t=1,\ldots, T$
\begin{equation} \label{aposteriori}
\begin{aligned}
z_{t}^{1+} &=  z_t^{1-} + L_t^1 \left(y_t^1 - C_t^1 z_t^{1-}\right), \\
z_{t}^{2+} &=  z_t^{2-} + L_t^2 \left(y_t^2 - C_t^2 z_t^{2-}\right), \\
\Sigma_{t}^+ &= \begin{bmatrix} I - L_t^1 C_t^1 & 0 \\ 0 & I - L_t^2 C_t^2 \end{bmatrix} \Sigma_t^- \begin{bmatrix} I - L_t^1 C_t^1 & 0 \\ 0 & I - L_t^2 C_t^2 \end{bmatrix}^\top + \begin{bmatrix} L_t^1  & 0 \\ 0 & L_t^2 \end{bmatrix} \begin{bmatrix} V_t^1  & 0 \\ 0 & V_t^2 \end{bmatrix} \begin{bmatrix} L_t^1  & 0 \\ 0 & L_t^2 \end{bmatrix}^\top,
\end{aligned}
\end{equation}
where the equilibrium filter gains $L_t^1$ and $L_t^2$ are obtained by jointly solving the coupled linear equations
\begin{equation} \label{optimalgains}
\begin{aligned}
L_t^1 &= \Sigma_t^{11-} C_t^{1\top} \left( C_t^{1} \Sigma_{t}^{11-} C_t^{1\top} + V_t^1 \right)^{-1} + \left(\Gamma_t^{11}\right)^{-1}  \Gamma_t^{12} \left(\Sigma_t^{12-\top} C_t^{1\top} + L_t^2 C_t^2 \Sigma_t^{12-\top} C_t^{1\top} \right), \\
L_t^2 &= \Sigma_t^{22-} C_t^{2\top} \left( C_t^{2} \Sigma_{t}^{22-} C_t^{2\top} + V_t^2 \right)^{-1} + \left(\Gamma_t^{22}\right)^{-1}  \Gamma_t^{12\top} \left(\Sigma_t^{12-} C_t^{2\top} + L_t^1 C_t^1 \Sigma_t^{12-} C_t^{2\top} \right) .
\end{aligned}
\end{equation}
\end{itemize}
The a priori and a posteriori updates can be combined to form a single update of the form \eqref{onestepfilters}, where
\begin{equation} \label{onestepparams}
\begin{aligned}
A_t^1 = A_t + B_t^2 K_t^2, \quad A_t^2 = A_t + B_t^1 K_t^1, \quad 
\bar B_t^1 = B_t^1, \quad \bar B_t^2 = B_t^2, \quad \bar L_t^1 = A_t^1 L_t^1, \quad \bar L_t^2 = A_t^2 L_t^2.
\end{aligned}
\end{equation}
\end{proposition}

\begin{proof}
The estimation error dynamics for player 1 are
\begin{equation}
\begin{aligned}
    e_{t+1}^1 &= x_{t+1} - z_{t+1}^1 \\
    & = A_t x_t + B_t^1 u_t^1 + B_t^2 u_t^2 + w_t - \left(A_t^1 z_t^1 + \bar B_t^1 u_t^1 + \bar L_t^1 \left(y_t^1 - C_t^1 z_t^1\right)\right) \\
    &= \left(A_t - A_t^1\right) x_t + \left(B_t^1 - \bar B_t^1\right) u_t^1 + B_t^2 u_t^2 + \left(A_t^1 - \bar L_t^1 C_t^1\right) e_t^1 + w_t - \bar L_t^1 v_t^1.
\end{aligned}
\end{equation}
Similarly, the estimation error dynamics for player 2 are
\begin{equation}
\begin{aligned}
    e_{t+1}^2 &= x_{t+1} - z_{t+1}^2 \\
    &= \left(A_t - A_t^2\right) x_t + \left(B_t^2 - \bar B_t^2\right) u_t^2 + B_1^1 u_t^1 + \left(A_t^2 - \bar L_t^2 C_t^2\right) e_t^2 + w_t - \bar L_t^2 v_t^1.
\end{aligned}
\end{equation}
Now using $A_t^1 = A_t + B_t^2 K_t^2$, $\bar B_t^1 = B_t^1$, $u_t^2 = K_t^2 z_t^2$, $A_t^2 = A_t + B_t^1 K_t^1$, $\bar B_t^2 = B_t^2$, and $u_t^1 = K_t^1 z_t^1$, we have
\begin{equation} \label{errors}
\begin{aligned}
    e_{t+1}^1 &= \left(A_t^1 - \bar L_t^1 C_t^1\right) e_t^1 - B_t^2 K_t^2 e_t^2 + w_t - \bar L_t^1 v_t^1, \\
    e_{t+1}^2 &= \left(A_t^2 - \bar L_t^2 C_t^2\right) e_t^2 - B_t^1 K_t^1 e_t^1 + w_t - \bar L_t^2 v_t^2.
\end{aligned}
\end{equation}
Taking  $\bar B_t^1 = B_t^1$ and  $\bar B_t^2 = B_t^2$ makes each player's estimation error independent of their own input. Taking $A_t^1 = A_t + B_t^2 K_t^2$ and $A_t^2 = A_t + B_t^1 K_t^1$ means
that each player uses their own estimate to approximate the unknown input signal of the opposing player (player $1$ approximates $K_t^2 z_t^2$ with $K_t^2 z_t^1$, and player $2$ approximates $K_t^1 z_t^1$ with $K_t^1 z_t^2$) according to Assumption \ref{asmp:belief_projection}. This couples the estimation error dynamics of each player and makes them dependent on the control strategies. Clearly, \eqref{errors} shows that $\mathbf{E} e_t^1 = \mathbf{E} e_t^2 = 0$ $\forall t$ (i.e., the estimates are unbiased), because the initialization \eqref{initialize} makes the initial estimates unbiased, and because the process and measurement noises are zero mean.

To derive the dynamics for the joint covariance and select equilibrium filter gain matrices, it is convenient to split the updates into a priori, which uses the (common knowledge) model, and a posteriori, which uses the (private) measurements. Defining $e_t^{i-} = x_t - z_t^{i-}$ and $e_t^{i-} = x_t - z_t^{i-}$ for $i=1,2$ and using the first two equations of \eqref{apriori}, we have
\begin{equation} \label{apriorierrors}
\begin{aligned}
    e_{t+1}^{1-} &= \left(A_t + B_t^2 K_t^2\right) e_t^{1+} - B_t^2 K_t^2 e_t^{2+} + w_t, \\
    e_{t+1}^{2-} &= \left(A_t + B_t^1 K_t^1\right) e_t^{2+} - B_t^1 K_t^1 e_t^{1+} + w_t.
\end{aligned}
\end{equation}
Substituting the above into $\Sigma_{t+1}^- = \mathbf{E}\begin{bmatrix} e_{t+1}^{1-} \\ e_{t+1}^{2-}  \end{bmatrix} \begin{bmatrix} e_{t+1}^{1-} \\ e_{t+1}^{2-} \end{bmatrix}^\top$ leads directly to the a priori covariance update in the third equation of  \eqref{apriori}. Note that the update at time $0$ is independent of the inputs at time $0$, since the initializations are common knowledge. However, after each player receives their first private measurement, the error dynamics become coupled, as noted above.

Defining the a posteriori errors $e_t^{i+} = x_t - z_t^{i+}$ and $e_t^{i-} = x_t - z_t^{i+}$ for $i=1,2$ and using the first two equations of \eqref{aposteriori}
\begin{equation} \label{apriorierrors}
\begin{aligned}
    e_{t}^{1+} &= \left(I - L_t^1 C_t^1\right) e_t^{1-} - L_t^1 v_t^1, \\
    e_{t}^{2+} &= \left(I - L_t^2 C_t^2\right) e_t^{2-} - L_t^2 v_t^2.
\end{aligned}
\end{equation}
Substituting into $\Sigma_{t}^+ = \mathbf{E}\begin{bmatrix} e_{t}^{1+} \\ e_{t}^{2+}  \end{bmatrix} \begin{bmatrix} e_{t}^{1+} \\ e_{t}^{2+} \end{bmatrix}^\top$ leads directly to the a posteriori covariance update in the third equation of \eqref{aposteriori}.

Now we select the filter gain matrices to optimize the weighted error covariance. Since the a priori update is independent of the filter gain matrices, it is equivalent to select them to optimize the expected total a posteriori estimation error. In particular, at time $t$ the filter gains can be selected to optimize
\begin{equation} \label{eq:cost_filter_gain}
\begin{aligned}
\mathcal{J}(L_t^1, L_t^2) = \mathbf{E} \begin{bmatrix}
     e_t^{1+} \\ e_t^{2+} 
 \end{bmatrix}^\top
 \begin{bmatrix}
     \Gamma_t^{11} & \Gamma_t^{12}  \\
      \Gamma_t^{12\top} & \Gamma_t^{22} 
 \end{bmatrix}
  \begin{bmatrix}
     e_t^{1+} \\ e_t^{2+}
 \end{bmatrix} = \mathbf{trace}\left( \Gamma_t \Sigma_t^+ \right) .
\end{aligned}
\end{equation}
Expanding, taking the derivative with respect to $L_t^1$ and $L_t^2$, and setting to zero gives
\begin{eqnarray}
\begin{aligned}
\frac {\partial \mathcal{J}}{\partial L_t^1} &=  \Gamma_t^{11} \left( L_t^1  \left( C_t^{1} \Sigma_{t}^{11-} C_t^{1T} + V_t^1 \right) -  \Sigma_t^{11-} C_t^{1\top} \right) + \Gamma_t^{12} \left(L_t^2 C_t^2 \Sigma_t^{12-\top} C_t^{1\top} - \Sigma_t^{12-\top} C_t^{1\top}\right) = 0, \\
\frac {\partial \mathcal{J}}{\partial L_t^2} &=  \Gamma_t^{22} \left( L_t^2  \left( C_t^{2} \Sigma_{t}^{22-} C_t^{2T} + V_t^2 \right) -  \Sigma_t^{22-} C_t^{2\top} \right) + \Gamma_t^{12\top} \left(L_t^1 C_t^1 \Sigma_t^{12-} C_t^{2\top} - \Sigma_t^{12-} C_t^{2\top}\right) = 0  .
\end{aligned}
\end{eqnarray}
Rearranging the forms above yields the coupled equations \eqref{optimalgains}, which can be solved jointly for $L_t^1$ and $L_t^2$ to get the filter gains explicitly in terms of the (a priori) joint covariance, weight matrix, and model parameters. Finally, it is straightforward to combine the a priori and a posteriori updates to form a single update of the form \eqref{onestepfilters}, with the parameters specified in \eqref{onestepparams}.
\end{proof}

\begin{discussion}

Each player can independently propagate the joint error covariance forward using the (common knowledge) model information and compute both filter gain matrices. However, the state estimates for each player are computed separately based on their private information  (measurement and control input histories). 

The joint error covariance trajectory $\boldsymbol{\Sigma} = [\Sigma_0^+, \Sigma_1^-, \ldots, \Sigma_T^-]$ and filter parameters $\mathbf{F} = \{(A_t^1, \bar B_t^1, \bar L_t^1), (A_t^2, \bar B_t^2, \bar L_t^2)\}_{t=1}^{T-1}$ can be computed with the forward recursion specified in Proposition \ref{prop1:partialinfofilters} based on the common knowledge model parameters $\mathcal{M}$, the gain matrices $\mathbf{K} = [K_0^1, K_0^2, K_1^1, K_1^2, \ldots, K_{T-1}^1, K_{T-1}^2]$ and weight matrices $\boldsymbol{\Gamma} = [\Gamma_0, \Gamma_1, \ldots, \Gamma_T]$. We express this forward recursion compactly as $$\left(\boldsymbol{\Sigma}, \mathbf{F}\right) = \mathtt{Forward}\left(\mathcal{M}, \mathbf{K}, \boldsymbol{\Gamma}\right).$$ In the standard Kalman filter, the forward recursion does not depend on the control gain or filter weight matrix.
\end{discussion}

\subsection{Backward Belief-State Feedback Strategies with Partial and Asymmetric Information} \label{partialinfoDP}
We now consider the dynamic game to characterize the equilibrium feedback strategies when the filter parameters for propagating belief states are fixed.  Solving the control problem involves a backward dynamic programming recursion to propagate expected costs and compute equilibrium strategies. 
We assume that each player uses a filter of the form \eqref{onestepfilters},
but with the parameters $(A_t^1, \bar B_t^1, \bar L_t^1)$ and $(A_t^2, \bar B_t^2, \bar L_t^2)$ for $t=0,\ldots, T-1$ fixed to arbitrary values and without the inputs specified yet. The joint state and estimation error dynamics are then given by
\begin{equation}
\begin{aligned}
    &\begin{bmatrix}
     x_{t+1} \\ e_{t+1}^1 \\ e_{t+1}^2 
    \end{bmatrix} = 
    \begin{bmatrix}
     A_t & 0 & 0 \\ A_t - A_t^1 & A_t^1 - L_t^1 C_t^1 &0 \\ A_t - A_t^2 & 0 & A_t^2 - L_t^2 C_t^2
    \end{bmatrix} 
    \begin{bmatrix}
     x_{t} \\ e_{t}^1 \\ e_{t}^2 
    \end{bmatrix} + 
    \begin{bmatrix}
     B_t^1 & B_t^2  \\ B_t^1 - \bar B_t^1 & B_t^2 \\ B_t^1 & B_t^2 - \bar B_t^2
    \end{bmatrix} 
    \begin{bmatrix}
     u_{t}^1 \\ u_{t}^2 
    \end{bmatrix} + 
    \begin{bmatrix}
     I & 0 & 0 \\ I &  -L_t^1  & 0 \\ I & 0 & - L_t^2
    \end{bmatrix} \begin{bmatrix} w_{t} \\ v_{t}^1 \\ v_{t}^2  \end{bmatrix}.  \\
\end{aligned}
\end{equation}
Denoting the augmented state $\mathbf{{X}}_t = \begin{bmatrix} x_{t} \\ e_{t}^1 \\ e_{t}^2  \end{bmatrix}\in \mathbf{R}^{3n}$ and augmented disturbance $\mathbf{w}_t = \begin{bmatrix} w_{t} \\ v_{t}^1 \\ v_{t}^2  \end{bmatrix} \in \mathbf{R}^{n+p_1 + p_2}$ with covariance matrix $\mathbf{W}_t =  \begin{bmatrix}
     W_t & 0 & 0 \\ 0 &  V_t^1  & 0 \\ 0 & 0 & V_t^2
    \end{bmatrix}$, these dynamics can be expressed compactly as
\begin{equation}
\begin{aligned}
    \mathbf{X}_{t+1} &= \mathbf{A}_t \mathbf{X}_t +  \mathbf{B}_t^1 u_t^1 +  \mathbf{B}_t^2  u_t^2  + \mathbf{G}_t \mathbf{w}_t \\
    &= \mathbf{f}_t \left( \mathbf{X}_t , u_t^1, u_t^2, \mathbf{w}_t \right).
\end{aligned}
\end{equation}
The stage cost functions can be expressed in terms of the augmented state $\mathbf{X}_t$ as
\begin{equation} \label{stagecostaugmented}
\begin{aligned}
    \ell_t \left(\mathbf{X}_t, u_t^1, u_t^2\right) &= \begin{bmatrix} \mathbf{X}_{t} \\ u_{t}^1 \\ u_{t}^2  \end{bmatrix}^\top
    \begin{bmatrix}
      \mathbf{Q}_t & 0 & 0 \\ 0 &  R_t  & 0 \\ 0 & 0 & S_t
    \end{bmatrix} \begin{bmatrix} \mathbf{X}_{t} \\ u_{t}^1 \\ u_{t}^2  \end{bmatrix}, \\
    \ell_T \left(\mathbf{X}_T\right) &= \mathbf{X}_T^\top \mathbf{Q}_T \mathbf{X}_T,
\end{aligned}
\end{equation}
where $\mathbf{Q}_t = \begin{bmatrix}
      Q_t & 0 & 0 \\ 0 &  0  & 0 \\ 0 & 0 & 0
\end{bmatrix}$.

Our goal in this section is to find the feedback gains $\{K_t^{1*}, K_t^{2*}\}_{t=0}^{T-1}$ of a saddle-point equilibrium with strategies $u_t^{1*} = K_t^{1*} z_t^1$ and $u_t^{2*} = K_t^{2*} z_t^2$ for the following objective. To achieve this goal, we first establish a general dynamic programming framework for computing feedback strategies of the form $u_t^1 = \pi_t(z_t^1)$ and $u_t^2 = \mu_t(z_t^2)$ for $t=0,1\ldots,T-1$, under Assumptions \ref{asmp:common_knowledge} and \ref{asmp:belief_projection}. We then specialize this framework to the linear-quadratic setting to derive explicit expressions for the gains $K_t^{1*}$ and $K_t^{2*}$.

Since the control strategies of each player depend on their state estimates, the total equilibrium cost can be expressed as a function of the initial state and initial estimation errors (using $e_t^i = x_t - z_t^i$). We define the total equilibrium cost using a saddle-point equilibrium strategy pair ($\pi^{*}$, $\mu^{*}$) with $\pi^{*} = (\pi_0^{*}, \pi_1^{*}, \ldots, \pi_{T-1}^{*})$ and $\mu^{*} = (\mu_0^{*}, \mu_1^{*}, \ldots, \mu_{T-1}^{*})$ for a given initial state and initial estimation errors $\mathbf{X}_0 = (x_0, e_0^1, e_0^2)$ by
\begin{equation} \label{totalcostLQR-initial}
    J_0^*(\mathbf{X}_0) = \mathbf{E}_{w_t, v_t^1, v_t^2} \left[ \sum_{t=0}^{T-1} \ell_t \left(\mathbf{X}_t, \pi_t^{*}\left( z_t^1\right), \mu_t^{*} \left( z_t^2 \right) \right) +  \ell_T \left(\mathbf{X}_T\right) \right].
\end{equation}

In analogy to the Principle of Optimality in optimal control theory, we state a ``Principle of Subgame Equilibrium'' for dynamic games. This is related to the notion of \emph{subgame perfect Nash equilibria} developed by Reinhard Selten \cite{selten1965spieltheoretische} and the notion of \emph{time consistency} \cite{bacsar1998dynamic}.

\begin{lemma}[Principle of Subgame Equilibrium in Dynamic Games] \label{subgameprinciple}
Suppose ($\pi^{*}, \mu^{*}$) is a saddle-point equilibrium strategy pair for the dynamic game. Consider the subgame at time $t$ in augmented state $\mathbf{X}_t$
\begin{equation} \label{subgame}
\min_{\pi^t}  \max_{\mu^t} \mathbf{E}_{w_t, v_t^1, v_t^2} \left[ \sum_{\tau=t}^{T-1} \ell_\tau \left(\mathbf{X}_\tau, \pi_\tau\left(z_\tau^1\right), \mu_\tau\left(z_\tau^2\right) \right) +  \ell_T \left( \mathbf{X}_T \right) \right],
\end{equation}
where the optimizations are over tail strategies $\pi^t = (\pi_t, \pi_{t+1}, \ldots, \pi_{T-1})$ and $\mu^t = (\mu_t, \mu_{t+1}, \ldots, \mu_{T-1})$ of the state estimates. Then the tail part $(\pi^{t*}, \mu^{t*})$ of the equilibrium $(\pi^*, \mu^*)$ is an equilibrium strategy for the subgame \eqref{subgame}.
\end{lemma}

\begin{proof}
Suppose the tail part ($\pi^{t*}, \mu^{t*}$) of the equilibrium ($\pi^{*}, \mu^{*}$) were not an equilibrium strategy for the subgame \eqref{subgame}. Then at least one of the players could improve their cost by modifying their strategy for the subgame. 
This contradicts that the tail strategies form part an equilibrium for the full time horizon, since one player could improve their cost by modifying their tail strategy.
\end{proof}

Lemma \ref{subgameprinciple} provides the foundation for backward induction. We now formalize how this principle translates into a computational dynamic programming algorithm for computing state estimate equilibrium strategies under Assumptions \ref{asmp:common_knowledge} and \ref{asmp:belief_projection}.

\begin{lemma}[Dynamic Programming for Equilibrium Strategies] \label{DPalgorithm}
Under Assumptions \ref{asmp:common_knowledge} and \ref{asmp:belief_projection}, the equilibrium cost function $J_0^*(\mathbf{X}_0)$ is given by the last step of the following backwards dynamic programming recursion
\begin{itemize}
\item \textbf{Initialize:}
\begin{equation} \label{initializeDP}
J_T(\mathbf{X}_T) = \ell_T \left(\mathbf{X}_T\right).
\end{equation}

\item \textbf{Dynamic Programming:} For $t=T-1,\ldots, 0$, the cost functions are updated according to
\begin{equation} \label{DPrecursion}
J_t \left(\mathbf{X}_t\right) = \mathbf{E}_{\mathbf{w}_t} \left[ \ell_t \left( \mathbf{X}_t, u_{t}^{1*} , u_{t}^{2*} \right)  + J_{t+1}\left( \mathbf{f}_t \left( \mathbf{X}_t , u_{t}^{1*}, u_{t}^{2*}, \mathbf{w}_t \right) \right) \right],
\end{equation}
where the equilibrium policies are given by
\begin{equation}
\begin{aligned}
    u_t^{1*} &= \pi_t^*\left( z_t^1 \right) = \arg \min_{u_t^1} \mathbf{E} \left[ \mathcal{Q}_t\left(\mathbf{X}_t, u_t^1, u_t^{2*}\right) \mid z_t^1 \right], \\
    u_t^{2*} &= \mu_t^*\left( z_t^2 \right) = \arg \max_{u_t^2} \mathbf{E} \left[ \mathcal{Q}_t\left(\mathbf{X}_t, u_t^{1*}, u_t^2\right) \mid z_t^2 \right],
\end{aligned}
\end{equation}
where
\begin{equation}
\mathcal{Q}_t\left( \mathbf{X}_t, u_t^1, u_t^2\right) = \mathbf{E}_{\mathbf{w}_t} \left[ \ell_t \left(\mathbf{X}_t, u_t^1, u_t^2 \right)  + J_{t+1}\left(\mathbf{f}_t\left( \mathbf{X}_t , u_t^1, u_t^2, \mathbf{w}_t \right)\right) \right].  
\end{equation}
\end{itemize}
\end{lemma}

\begin{proof}
For $t = 0,\ldots,T$, we define the equilibrium cost-to-go functions for the subgame using tail strategy pairs $\pi^{t} = (\pi_t, \pi_{t+1}, \ldots, \pi_{T-1})$ and $\mu^{t} = (\mu_t, \mu_{t+1}, \ldots, \mu_{T-1})$ starting from $\mathbf{X}_t = \begin{bmatrix} x_t \\ e_t^1 \\ e_t^2 \end{bmatrix}$ by
\begin{equation} \label{totalcostLQR-initial}
    J_t^*\left(\mathbf{X}_t\right) = \min_{\pi^{t}}  \max_{\mu^{t}} \mathbf{E}_{\mathbf{w}_t} \left[ \sum_{\tau=t}^{T-1} \ell_\tau \left(\mathbf{X}_\tau, \pi_\tau\left(z_\tau^1\right), \mu_\tau\left(z_\tau^2\right) \right) +  \ell_T \left(\mathbf{X}_T\right) \right].
\end{equation}
The base case for induction is established by defining $J_T^*(\mathbf{X}_T) := \ell_T (\mathbf{X}_T) =  J_T(\mathbf{X}_T)$.

For the induction hypothesis, suppose for some $t$ that $J_{t+1}(\mathbf{X}_t) = J_{t+1}^*(\mathbf{X}_{t+1})$. Then
\begin{subequations} \label{DPproof}
\begin{align}
J_t^*(\mathbf{X}_t) &= \min_{\pi_t} \max_{\mu_t} \mathbf{E}_{\mathbf{w}_t} \left[  \ell_t \left(\mathbf{X}_t, \pi_t\left(z_t^1\right), \mu_t\left(z_t^2\right) \right) + \min_{\pi^{t+1}} \max_{\mu^{t+1}} \mathbf{E}_{\mathbf{w}_{t+1}} \sum_{\tau=t+1}^{T-1} \ell_\tau \left(\mathbf{X}_\tau, \pi_\tau\left(z_\tau^1\right), \mu_\tau\left(z_\tau^2\right) \right) +  \ell_T \left(\mathbf{X}_T\right) \right]\label{DPproof-a} \\
&= \min_{\pi_t} \max_{\mu_t} \mathbf{E}_{\mathbf{w}_t} \left[ \ell_t \left(\mathbf{X}_t, \pi_t\left(z_t^1\right), \mu_t \left( z_t^2 \right) \right) + J_{t+1}^*\left( \mathbf{X}_{t+1} \right) \right] \label{DPproof-b}\\
&= \min_{\pi_t} \max_{\mu_t} \mathbf{E}_{\mathbf{w}_t} \left[  \ell_t \left( \mathbf{X}_t, \pi_t \left( z_t^1 \right), \mu_t \left( z_t^2 \right) \right) + J_{t+1}^* \left( \mathbf{f}_t \left( \mathbf{X}_t , \pi_t \left( z_t^1 \right), \mu_t \left( z_t^2 \right), \mathbf{w}_t \right) \right) \right]\label{DPproof-c}  \\
&= \min_{\pi_t} \max_{\mu_t} \mathbf{E}_{\mathbf{w}_t} \left[  \ell_t \left( \mathbf{X}_t, \pi_t \left( z_t^1 \right), \mu_t \left( z_t^2 \right) \right) + J_{t+1} \left( \mathbf{f}_t \left( \mathbf{X}_t , \pi_t \left( z_t^1 \right), \mu_t \left( z_t^2 \right), \mathbf{w}_t \right) \right) \right]\label{DPproof-d}  \\ 
&= \mathbf{E}_{\mathbf{w}_t} \left[ \ell_t \left( \mathbf{X}_t, u_{t}^{1*} , u_{t}^{2*} \right)  + J_{t+1}\left( \mathbf{f}_t \left( \mathbf{X}_t , u_{t}^{1*}, u_{t}^{2*}, \mathbf{w}_t \right) \right) \right] \label{DPproof-e} \\
&= J_{t} \left( \mathbf{X}_t \right)\label{DPproof-f},
\end{align}
\end{subequations}
where the equilibrium input strategies jointly satisfy
\begin{equation}
\begin{aligned}
    u_t^{1*} &= \pi_t^*\left( z_t^1 \right) = \arg \min_{u_t^1} \mathbf{E} \left[ \mathcal{Q}_t\left(\mathbf{X}_t, u_t^1, u_t^{2*}\right) \mid z_t^1 \right], \\
    u_t^{2*} &= \mu_t^*\left( z_t^2 \right) = \arg \max_{u_t^2} \mathbf{E} \left[ \mathcal{Q}_t\left(\mathbf{X}_t, u_t^{1*}, u_t^2\right) \mid z_t^2 \right],
\end{aligned}
\end{equation}
and
\begin{align} \label{qfunction}
\mathcal{Q}_t\left( \mathbf{X}_t, u_t^1, u_t^2\right) = \mathbf{E}_{\mathbf{w}_t} \left[ \ell_t \left(\mathbf{X}_t, u_t^1, u_t^2 \right)  + J_{t+1}\left(\mathbf{f}_t\left( \mathbf{X}_t , u_t^1, u_t^2, \mathbf{w}_t \right)\right) \right].  
\end{align}
In \eqref{DPproof-a}, the Principle of Subgame Equilibrium is used to separate optimization of the current strategies at time $t$ from the tail strategies for the subgame from time $t+1$ to the end of the horizon. In \eqref{DPproof-b}, \eqref{DPproof-c}, and \eqref{DPproof-d}, the definition of the equilibrium cost-to-go functions, the dynamics function, and the induction hypothesis are used, respectively. In \eqref{DPproof-e}, the optimization over feedback strategies is converted to pointwise optimizations over the respective inputs over all possible values of the augmented state, conditioned on the state estimates for each player.
\end{proof}

Lemma \ref{DPalgorithm} establishes the general dynamic programming framework for computing state estimate feedback equilibrium strategies and cost functions. We now specialize this framework to the linear-quadratic setting. The following proposition shows that under this setting, the cost-to-go functions maintain a quadratic structure throughout the backward recursion and provide explicit expressions for computing the linear saddle-point equilibrium strategies.

\begin{proposition} \label{DPalgorithmLQ}
Under the linear-quadratic setting, equilibrium linear state estimate feedback strategies and their corresponding quadratic equilibrium cost functions can be computed via the following backward dynamic programming recursion:
\begin{itemize}
\item \textbf{Initialize:}
\begin{equation} \label{initializeDP}
P_T = \mathbf{Q}_T, \quad r_T = 0.
\end{equation}

\item \textbf{Dynamic Programming:} For $t=T-1,\ldots,0$, define the matrix
\begin{equation} \label{DPqfunction}
\begin{aligned}
\mathcal{Q}_t (P_{t+1}) =  \begin{bmatrix}
     \mathcal{Q}_t^{00} & \mathcal{Q}_t^{01} & \mathcal{Q}_t^{02}  \\
      \mathcal{Q}_t^{01\top} & \mathcal{Q}_t^{11} & \mathcal{Q}_t^{12} \\
      \mathcal{Q}_t^{02\top} & \mathcal{Q}_t^{12\top} & \mathcal{Q}_t^{22} 
 \end{bmatrix} = \begin{bmatrix}
     \mathbf{Q}_t + \mathbf{A}_t^\top P_{t+1} \mathbf{A}_t & \mathbf{A}_t^\top P_{t+1} \mathbf{B}_t^1 & \mathbf{A}_t^\top P_{t+1} \mathbf{B}_t^2  \\
      \mathbf{B}_t^{1\top} P_{t+1} \mathbf{A}_t & R_t + \mathbf{B}_t^{1\top} P_{t+1} \mathbf{B}_t^1 & \mathbf{B}_t^{1\top} P_{t+1} \mathbf{B}_t^2 \\
      \mathbf{B}_t^{2\top} P_{t+1} \mathbf{A}_t &  \mathbf{B}_t^{2\top} P_{t+1} \mathbf{B}_t^1 & S_t +  \mathbf{B}_t^{2\top} P_{t+1} \mathbf{B}_t^2 
 \end{bmatrix}.
\end{aligned}
\end{equation}
Suppose $ \mathcal{Q}_t^{11} \succ 0$ and $ \mathcal{Q}_t^{22} \prec 0$, $\forall t$. Then the saddle-point equilibrium strategies $u_t^{1*} = K_t^{1*} z_t^1$ and $u_t^{2*} = K_t^{2*} z_t^2$ have feedback gains given by the left $m_1 \times n$ and $m_2 \times n$ blocks, respectively, of the matrices
\begin{equation} \label{DPgains}
\begin{aligned}
   \bar K_t^{1*} &= \left(\mathcal{Q}_t^{11} - \mathcal{Q}_t^{12} \left( \mathcal{Q}_t^{22} \right)^{-1}  \mathcal{Q}_t^{12\top}\right)^{-1} \times \left(\mathcal{Q}_t^{12} \left( \mathcal{Q}_t^{22} \right)^{-1}  \mathcal{Q}_t^{02\top} -  \mathcal{Q}_t^{01\top} \right), \\
   \bar K_t^{2*} &= \left(\mathcal{Q}_t^{22} - \mathcal{Q}_t^{12\top} \left( \mathcal{Q}_t^{11} \right)^{-1} \mathcal{Q}_t^{12} \right)^{-1} \times \left( \mathcal{Q}_t^{12\top} \left( \mathcal{Q}_t^{11} \right)^{-1}  \mathcal{Q}_t^{01\top} -  \mathcal{Q}_t^{02\top} \right),
\end{aligned}
\end{equation}
where $\bar{K}_t^{1*} = [K_t^{1*} \ \ 0_{m_1 \times n} \ \ 0_{m_1 \times n}] $ and $\bar{K}_t^{2*} = [K_t^{2*} \ \ 0_{m_2 \times n} \ \ 0_{m_2 \times n} ] $. Note that once the equilibrium augmented state feedback gains $\bar K_t^{1*}$ and $\bar K_t^{2*}$ are determined, then the state estimate feedback gains $K_t^{1*}$ and $K_t^{2*}$ are given by the left $m_1 \times n$ and $m_2 \times n$ blocks of $\bar K_t^{1*}$ and $\bar K_t^{2*}$.

Furthermore, the equilibrium cost-to-go functions take the quadratic form $J_t(\mathbf{X}_t) = \mathbf{X}_t^\top P_t \mathbf{X}_t + r_t$, with the cost parameters satisfying the backward recursion for $t=T-1,\ldots,0$,
\begin{equation} \label{DP}
\begin{aligned}
P_t &=  \begin{bmatrix} I_{3n} \\ \bar{\mathbf{K}}_t^{1*}  \\ \bar{\mathbf{K}}_t^{2*}   \end{bmatrix}^\top  \mathcal{Q}_t \left( P_{t+1} \right) \begin{bmatrix} I_{3n} \\ \bar{\mathbf{K}}_t^{1*}  \\ \bar{\mathbf{K}}_t^{2*}  \end{bmatrix}, \\
   r_t &=  r_{t+1} + \mathbf{tr} \left( P_{t+1} \mathbf{G}_t \mathbf{W}_t \mathbf{G}_t^T \right) .
\end{aligned}
\end{equation}
where $\bar{\mathbf{K}}_t^{1*} = [K_t^{1*}   \ \ -K_t^{1*} \ \ 0_{m_1 \times n}] $ and $\bar{\mathbf{K}}_t^{2*} = [K_t^{2*}  \ \ 0_{m_2 \times n} \ \ -K_t^{2*} ] $.
\end{itemize}
\end{proposition}

\begin{proof}
We show via induction that the equilibrium cost-to-go functions are quadratic in the augmented state $\mathbf{X}_t$ (i.e., they are jointly quadratic in the state and estimation errors). Along the way, we show that under Assumption \ref{asmp:belief_projection} the saddle-point equilibrium control strategies are also linear in the augmented state $\mathbf{X}_t$.

\textbf{Base case:} The initialization of the backward recursion \eqref{initializeDP} with the quadratic stage cost function \eqref{stagecostaugmented} estabilishes $J_T(\mathbf{X}_T) = \mathbf{X}_T^\top P_T \mathbf{X}_T$.

\textbf{Inductive step:} Suppose $J_{t+1}(\mathbf{X}_{t+1}) = \mathbf{X}_{t+1}^\top P_{t+1} \mathbf{X}_{t+1} + r_{t+1}$. Substituting this into the $\mathcal{Q}$-function in \eqref{qfunction} and evaluating the expectation with respect to the augmented disturbance yields the jointly quadratic function
\begin{equation} \label{qfunctioncost}
\begin{aligned}
\mathcal{Q}_t \left( \mathbf{X}_t, u_t^1, u_t^2 \right) &=\begin{bmatrix} \mathbf{X}_{t} \\ u_{t}^1 \\ u_{t}^2  \end{bmatrix}^\top
    \begin{bmatrix}
      \mathbf{Q}_t & 0 & 0 \\ 0 &  R_t  & 0 \\ 0 & 0 & S_t
    \end{bmatrix} \begin{bmatrix} \mathbf{X}_{t} \\ u_{t}^1 \\ u_{t}^2  \end{bmatrix}  +  \mathbf{E}_{\mathbf{w}_t}  J_{t+1}\left( \mathbf{A}_t \mathbf{X}_t +  \mathbf{B}_t^1 u_t^1 +  \mathbf{B}_t^2  u_t^2  + \mathbf{G}_t \mathbf{w}_t \right) \\
 &=  \begin{bmatrix} \mathbf{X}_{t} \\ u_{t}^1 \\ u_{t}^2  \end{bmatrix}^\top  \mathcal{Q}_t(P_{t+1})   \begin{bmatrix} \mathbf{X}_{t} \\ u_{t}^1 \\ u_{t}^2  \end{bmatrix} + \mathbf{tr} \left( P_{t+1} \mathbf{G}_t \mathbf{W}_t \mathbf{G}_t^T \right) + r_{t+1}.
 \end{aligned}
\end{equation}
Since players do not observe $\mathbf{X}_t$ directly, they optimize based on their state estimates. By Lemma \ref{DPalgorithm}, the equilibrium strategies satisfy
\begin{equation}
\begin{aligned}
    u_{t}^{1*} &= \pi_t^* \left( z_t^1 \right) = \arg \min_{u_t^1} \mathbf{E} \left[ \mathcal{Q}_t \left(\mathbf{X}_t, u_t^1, u_t^{2*} \right) \mid z_t^1 \right], \\
    u_{t}^{2*} &= \mu_t^* \left( z_t^2 \right) = \arg \max_{u_t^2} \mathbf{E} \left[ \mathcal{Q}_t \left(\mathbf{X}_t, u_t^{1*}, u_t^{2} \right) \mid z_t^2 \right].
\end{aligned}
\end{equation}
Taking derivatives with respect to the inputs, provided that $\mathcal{Q}_t^{11} \succ 0$ and $\mathcal{Q}_t^{22} \prec 0$, we get a saddle-point equilibrium when the inputs jointly satisfy
\begin{equation}\label{eqn:firstorder}
\begin{aligned}
    u_{t}^{1*} &= \left(\mathcal{Q}_t^{11}\right)^{-1} \left( \mathcal{Q}_t^{01\top} \mathbf{E}\left[\mathbf{X}_t \mid z_{t}^1 \right] + \mathcal{Q}_t^{12} \mathbf{E}\left[u_{t}^{2*} \mid z_{t}^1 \right] \right), \\
    u_{t}^{2*} &= \left(\mathcal{Q}_t^{22}\right)^{-1} \left( \mathcal{Q}_t^{02\top} \mathbf{E}\left[\mathbf{X}_t \mid z_{t}^2 \right] + \mathcal{Q}_t^{12\top} \mathbf{E}\left[u_{t}^{1*} \mid z_{t}^2 \right] \right).
\end{aligned}
\end{equation}
Under Assumption \ref{asmp:belief_projection}, we have $\mathbf{E}\left[\mathbf{X}_t \mid z_t^i\right] = \begin{bmatrix} z_t^i \\ 0 \\ 0 \end{bmatrix}$ because of $\mathbf{E}\left[x_t \mid z_t^i\right] = z_t^i$ and $\mathbf{E}\left[e_t^j \mid z_t^i\right] = 0$ for any $i,j\in\{1,2\}$. Moreover, each player uses its state estimate to approximate $\mathbf{E}\left[\mathbf{E}\left[\mathbf{X}_t \mid z_t^j\right] \mid z_t^i\right]$ by $\mathbf{E}\left[\mathbf{X}_t \mid z_t^i\right]$. This motivates seeking equilibrium strategies linear in augmented state estimates: $u_t^{1*}=\bar K_t^{1*} \mathbf{E}\left[\mathbf{X}_t \mid z_t^1\right]$ and $u_t^{2*}=\bar K_t^{2*} \mathbf{E}\left[\mathbf{X}_t \mid z_t^2\right]$. Substituting these strategies into \eqref{eqn:firstorder} gives
\begin{equation} 
\begin{aligned}
     \bar K_t^{1*} \mathbf{E}\left[\mathbf{X}_t \mid z_{t}^1 \right] &= \left( \mathcal{Q}_t^{11} \right)^{-1} \left( \mathcal{Q}_t^{01\top} \mathbf{E}\left[\mathbf{X}_t \mid z_{t}^1 \right] + \mathcal{Q}_t^{12} \mathbf{E}\left[ \bar K_t^{2*} \mathbf{E} \left[\mathbf{X}_t \mid z_{t}^2 \right] \mid z_{t}^1 \right] \right), \\
     \bar K_t^{2*} \mathbf{E}\left[\mathbf{X}_t \mid z_{t}^2 \right] &= \left( \mathcal{Q}_t^{22} \right)^{-1} \left( \mathcal{Q}_t^{02\top} \mathbf{E}\left[\mathbf{X}_t \mid z_{t}^2 \right] + \mathcal{Q}_t^{12\top} \mathbf{E}\left[ \bar K_t^{1*} \mathbf{E}\left[\mathbf{X}_t \mid z_{t}^1 \right]  \mid z_{t}^2 \right] \right).
\end{aligned}
\end{equation}
Then the augmented state feedback gains $\bar K^{1*}$ and $\bar K^{2*}$ jointly satisfy
\begin{equation}
\begin{aligned}
     \bar K_t^{1*} &= \left( \mathcal{Q}_t^{11} \right)^{-1} \left(\mathcal{Q}_t^{01\top} + \mathcal{Q}_t^{12}  \bar K_t^{2*} \right), \\
     \bar K_t^{2*} &= \left( \mathcal{Q}_t^{22} \right)^{-1} \left(\mathcal{Q}_t^{02\top} + \mathcal{Q}_t^{12\top} \bar K_t^{1*} \right),
\end{aligned}
\end{equation}
which can be solved jointly to give \eqref{DPgains}. Then the state estimate feedback gains $K_t^{1*}$ and $K_t^{2*}$ are given by the left $m_1 \times n$ and $m_2 \times n$ blocks of $\bar K_t^{1*}$ and $\bar K_t^{2*}$. Using $z_t^i = x_t - e_t^i$, these strategies can be expressed as linear in the augmented state as  $u_t^{1*} =  \bar{\mathbf{K}}_t^{1*} \mathbf{X}_t$ and $u_t^{2*} = \bar{\mathbf{K}}_t^{2*} \mathbf{X}_t$ where $\bar{\mathbf{K}}_t^{1*} = [K_t^{1*} \ \ -K_t^{1*} \ \ 0_{m_1 \times n}] $ and $\bar{\mathbf{K}}_t^{2*} = [K_t^{2*} \ \ 0_{m_2 \times n} \ \ -K_t^{2*} ] $. Finally, substituting these equilibrium strategies back into the cost expression \eqref{qfunctioncost} shows that the right-hand side of \eqref{DPrecursion} is quadratic and therefore $J_t(\mathbf{X}_t) = \mathbf{X}_{t}^\top P_{t} \mathbf{X}_{t} + r_{t}$, leading to the recursion for the equilibrium cost-to-go parameters given by \eqref{DP}.
\end{proof}

\begin{discussion}
In contrast to the single-player case, where the estimation error can be made independent of the control input, here each player's estimation error is affected by the opposing player's control input, another indication of the signaling phenomenon. This couples the estimation and control problems, since the equilibrium control gains depend on the estimator gains through the coupled dynamics of the augmented state.

The cost matrices $\mathbf{P} = [P_0, P_1, \ldots, P_T]$ and equilibrium feedback gains $\mathbf{K} = \{K_t^{1*}, K_t^{2*} \}_{t=0}^{T-1}$ can be computed with the backward recursion specified in Proposition \ref{DPalgorithmLQ} based on the common knowledge model $\mathcal{M}$ and filter parameters $\mathbf{F} = \{(A_t^1, \bar B_t^1, \bar L_t^1), (A_t^2, \bar B_t^2, \bar L_t^2)\}_{t=1}^{T-1}$. We express this backward recursion compactly as $$(\mathbf{P}, \mathbf{K}) = \mathtt{Backward}( \mathcal{M}, \mathbf{F}).$$ In the standard single-player LQG problem, the feedback control gain and the costs associated with the underlying state do not depend on the filter parameters.

When the (appropriately scalarized) penalty $S_t$ on the maximizing player approaches (negative) infinity, their gain matrices goes to zero, and the problem reduces to the standard single-player LQG problem with separation of estimation and control. 
\end{discussion}

\subsection{A Forward-Backward Algorithm for Computing Equilibrium Solutions} \label{sec:forward_backward_algorithm}
In summary, Sections \ref{partialinfofilter} and \ref{partialinfoDP} have shown that the problem of finding saddle-point equilibrium strategies for our linear quadratic stochastic dynamic game with partial and asymmetric information can be reduced to solving the pair of coupled forward and backward recursions described in Propositions \ref{prop1:partialinfofilters} and \ref{DPalgorithmLQ}. In particular, the forward recursion propagates belief states and computes equilibrium estimator gains for given values of the feedback gains and estimation error cost weight matrices (which are obtained from the lower right $2n\times 2n$ block of the cost matrices from Proposition \ref{DPalgorithmLQ}). The backward recursion propagates equilibrium cost matrices and computes equilibrium feedback control gaines for given values of the filter parameters. The coupled recursions are nonlinear matrix equations in the cost matrices $\mathbf{P}$, joint error covariances $\boldsymbol{\Sigma}$, feedback control gains $\mathbf{K}$ and filter parameters $\mathbf{F}$:
\begin{equation} \label{coupledforwardbackward}
\begin{aligned}
\left( \boldsymbol{\Sigma}, \mathbf{F} \right) &= \mathtt{Forward} \left( \mathcal{M}, \mathbf{K}, \boldsymbol{\Gamma} \right), \\
\left( \mathbf{P}, \mathbf{K} \right) &= \mathtt{Backward}\left( \mathcal{M}, \mathbf{F} \right),
\end{aligned}
\end{equation}
where the coupling comes from the expression for the equilibrium filter parameters in \eqref{optimalgains} and \eqref{onestepparams}, which depends on the control gains and error covariances, and the expression for the equilibrium feedback gains \eqref{DPgains}, which depends on the cost matrices and filter parameters. It is possible and natural to set the filter weight matrices $\boldsymbol{\Gamma}$ to the block of the cost matrices $\mathbf{P}$ corresponding to the estimation error. Then the forward recursion also depends on the cost matrices.
    
There are no existing algorithms to jointly solve these coupled recursions to our best knowledge. In principle, general numerical root finding algorithms could be applied. Some of these, such as Newton's method and variations, may be related to variations of other dynamic programming algorithms like policy iteration and will be the subject of future research. 

One natural approach is to iterate the forward and backward passes until convergence. In particular, we can initialize the feedback gain matrices $\mathbf{K}_0$ and estimation error weight matrix $\boldsymbol{\Gamma}_0$ and compute for $k=0,1,\ldots,$
\begin{equation}
\begin{aligned}
\left(\boldsymbol{\Sigma}_{k+1}, \mathbf{F}_{k+1}\right) &= \mathtt{Forward}\left(\mathcal{M}, \mathbf{K}_k, \boldsymbol{\Gamma}_k\right), \\
\left(\mathbf{P}_{k+1}, \mathbf{K}_{k+1} \right) &= \mathtt{Backward}\left(\mathcal{M}, \mathbf{F}_{k+1}\right).
\end{aligned}
\end{equation}
Clearly, if the iterates $(\boldsymbol{\Sigma}_{k+1}, \mathbf{F}_{k+1},\mathbf{P}_{k+1}, \mathbf{K}_{k+1})$ converge as $k\rightarrow \infty$, then we obtain a solution for the coupled forward and backward recursions \eqref{coupledforwardbackward}, which corresponds to an equilibrium strategy of the linear form \eqref{onestepfilters}, \eqref{linearestimatefeedback} for our stochastic dynamic game with partial and asymmetric information. 

To obtain an implementable algorithm for computing an approximate fixed point, we propose an iterative forward-backward algorithm in Algorithm \ref{algorithm:forward-backward}. 
\begin{algorithm}
\caption{Iterative Forward-Backward Algorithm}\label{algorithm:forward-backward}
\begin{algorithmic}[1]
    \Require Common knowledge model $\mathcal{M}$, initial gain matrices $\mathbf{K}$ and weight matrix $\boldsymbol{\Gamma}$, convergence tolerance $\epsilon$
    \While{$ \| \mathbf{K} - \mathbf{K}^- \| + \| \mathbf{F} - \mathbf{F}^- \| > \epsilon $}
        \State Store previous iterates: $\mathbf{K}^- = \mathbf{K}$, $\mathbf{F}^- = \mathbf{F}$ 
        \State Forward pass: $\left( \boldsymbol{\Sigma}, \mathbf{F} \right) = \mathtt{Forward}\left( \mathcal{M}, \mathbf{K}, \boldsymbol{\Gamma} \right)$ \;
        \State Backward pass: $\left( \mathbf{P}, \mathbf{K} \right) = \mathtt{Backward}\left( \mathcal{M}, \mathbf{F} \right) $ \;
    \EndWhile
    \Ensure Equilibrium gain matrices $\mathbf{K}$, filter parameters $\mathbf{F}$, cost matrices $\mathbf{P}$, error covariances $\boldsymbol{\Sigma}$
\end{algorithmic}
\end{algorithm}

Our implementation for this algorithm is available at \url{https://gitlab.com/scratch7473433/partialinfodynamicprogramming}. In Section \ref{sec:numerical_experiments}, we present numerical experiments that demonstrate convergence of the algorithm to equilibrium strategies for stochastic dynamic games with asymmetric and partial information. 

\section{Infinite Horizon LQG Dynamic Games with Partial and Asymmetric Information}\label{Section:InfHorizon}
We now consider an infinite-horizon LQG dynamic game with time-invariant common knowledge model  $\mathcal{M} = \{ A, B^1, B^2, W, C^1, C^2, V^1, V^2, Q, R, S \}$ and average steady-state cost objective
\begin{equation} \label{averagecostLQR}
\begin{aligned}
    J(u^1, u^2) \ = \lim_{T \rightarrow \infty} \frac{1}{T}\mathbf{E} \left[ \sum_{t=0}^{T-1} x_t^\top Q x_t + u_t^{1\top} R u_t^1 + u_t^{2\top} S u_t^2 \right].
\end{aligned}
\end{equation}
It is also possible to consider a discounted cost variation. Note that the initial state distribution is not included in the model since it does not impact the steady-state average-cost. In single-player optimal control problems, value iteration is a standard algorithm for computing optimal feedback control laws and value functions. Value iteration computes the right-hand side of the dynamic programming recursion, also known as the Bellman operator, from an initial value function until convergence. In this section, we describe an analogous value iteration-like algorithm for our infinite-horizon dynamic game with partial and asymmetric information based on iterating forward and backward operators until convergence.

\subsection{Forward Recursion Convergence}
If the forward belief propagation of state estimates in Proposition \ref{prop1:partialinfofilters} converges as $T \rightarrow \infty$ for a given pair of feedback gain matrices $(K^1, K^2)$ and weight matrix $\Gamma$, then the steady-state covariance $\Sigma := \lim_{t \rightarrow \infty} \Sigma^-_t$ satisfies
\begin{equation} \label{steadystatecovariance}
\Sigma =  \bar A \left(\Gamma \Sigma \Gamma^\top + \bar L \bar V \bar L^\top \right) \bar A^\top + \bar W ,
\end{equation}
where 
\begin{eqnarray} \label{steadystateblockmatrices}
\begin{aligned}
\Gamma = \begin{bmatrix} I - L^1 C^1 \hspace{-0.2cm} & 0 \\ 0 & I - L^2 C^2 \end{bmatrix}, \quad 
\bar A =  \begin{bmatrix} A + B^2 K^2 \hspace{-0.2cm} & - B^2 K^2 \\ -B^1 K^1 & A + B^1 K^1 \end{bmatrix}, \quad \bar L &= \begin{bmatrix} L^1  & 0 \\ 0 & L^2 \end{bmatrix}, \quad \bar W =  \begin{bmatrix} W  & 0 \\ 0 & W \end{bmatrix},\quad \bar V = \begin{bmatrix} V^1  & 0 \\ 0 & V^2 \end{bmatrix}.
\end{aligned}
\end{eqnarray}
The filters for state estimation take the time-invariant form
\begin{equation} \label{steadystatefilters}
\begin{aligned}
    z^1_{t+1} &= A^1 z_t^1 + \bar B^1 u_t^1 + \bar L^1 \left(y_t^1 - C^1 z_t^1\right), \\
    z^2_{t+1} &= A^2 z_t^2 + \bar B^2 u_t^2 + \bar L^2 \left(y_t^2 - C^2 z_t^2\right),
\end{aligned}
\end{equation}
where the steady-state filter parameters are given by 
\begin{equation} \label{steadystateoptimalgains}
\begin{aligned}
L^1 &= \Sigma^{11} C^{1\top} \left( C^{1} \Sigma^{11} C^{1\top} + V^1 \right)^{-1} + \left(\Gamma^{11}\right)^{-1}  \Gamma^{12} \left( \Sigma^{12\top} C^{1\top} + L^2 C^2 \Sigma^{12\top} C^{1\top} \right) \\
L^2 &= \Sigma^{22} C^{2\top} \left( C^{2} \Sigma_{t}^{22} C^{2\top} + V^2 \right)^{-1} + \left(\Gamma^{22}\right)^{-1}  \Gamma^{12\top} \left(\Sigma^{12-} C^{2\top} + L^1 C_t^1 \Sigma^{12-} C^{2\top} \right)
\end{aligned}
\end{equation}
with
\begin{eqnarray} \label{steadystateparams}
\begin{aligned}
A^1 = A + B^2 K^2, \quad A^2 = A + B^1 K^1, \quad \bar B^1 = B^1, \quad &\bar B^2 = B^2, \quad \bar L^1 = A^1 L^1,\quad\bar L^2 = A^2 L^2.
\end{aligned}
\end{eqnarray}

We define the \textbf{one-step forward operator} as the right-hand side of \eqref{steadystatecovariance} using the steady-state filter gains \eqref{steadystateoptimalgains} for a given pair of feedback control gains $K = (K^1, K^2)$ with common knowledge model $\mathcal{M}$, denoted by $\mathcal{F}(\mathcal{M}, \Sigma, K, \Gamma)$.

\subsection{Backward Recursion Convergence}
Likewise, if the quadratic component $P_t$ of the backward dynamic programming recursion in Proposition \ref{DPalgorithmLQ} converges for a given pair of time-invariant filter parameters $(A^1, \bar B^1, \bar L^1)$ and $(A^2, \bar B^2, \bar L^2)$, then the total cost matrix $P := \lim_{t \rightarrow -\infty} P_t$ satisfies 
\begin{equation} \label{Bellman}
\begin{aligned}
P &=  \begin{bmatrix} I_{3n} \\ \bar{\mathbf{K}}^{1}  \\ \bar{\mathbf{K}}^{2}   \end{bmatrix}^\top  \mathcal{Q} (P) \begin{bmatrix} I_{3n} \\ \bar{\mathbf{K}}^{1}  \\ \bar{\mathbf{K}}^{2}   \end{bmatrix},
\end{aligned}
\end{equation}
where the quantities $\bar{\mathbf{K}}^{1}, \bar{\mathbf{K}}^{2}$, and $\mathcal{Q} (P)$ are defined as before but with the time-invariant model parameters. In particular,
\begin{equation} \label{steadystateDPqfunction}
\begin{aligned}
&\mathcal{Q} (P) 
= \begin{bmatrix}
     \mathbf{Q} + \mathbf{A}^\top P \mathbf{A} & \mathbf{A}^\top P \mathbf{B}^1 & \mathbf{A}^\top P \mathbf{B}^2  \\
      \mathbf{B}^{1\top} P \mathbf{A} & R + \mathbf{B}^{1\top} P \mathbf{B}^1 & \mathbf{B}^{1\top} P \mathbf{B}^2 \\
      \mathbf{B}^{2\top} P \mathbf{A} &  \mathbf{B}^{2\top} P \mathbf{B}^1 & S +  \mathbf{B}^{2\top} P \mathbf{B}^2 
 \end{bmatrix}
\end{aligned}
\end{equation}
with
\begin{equation} \label{steadystateblockmatrices}
\begin{aligned}
\mathbf{A} = \begin{bmatrix}
     A & 0 & 0 \\ A - A^1 & A^1 - L^1 C^1 &0 \\ A - A^2 & 0 & A^2 - L^2 C^2
    \end{bmatrix}, \quad
\mathbf{B}^1 = \begin{bmatrix}
     B^1  \\ B^1 - \bar B^1  \\ B^1
    \end{bmatrix}, \quad
\mathbf{B}^2 = \begin{bmatrix}
      B^2  \\ B^2 \\  B^2 - \bar B^2
    \end{bmatrix}, \quad \mathbf{Q} = \begin{bmatrix}
      Q & 0 & 0 \\ 0 &  0  & 0 \\ 0 & 0 & 0
    \end{bmatrix}.
\end{aligned}
\end{equation}
The feedback strategies take the linear time-invariant form $u_t = K^{1} z_t^1$ and $u_t^2 = K^{2} z_t^2$, with the feedback gain matrices $K^{1}$ and $K^{2}$ given by the left $m_1 \times n$ and $m_2 \times n$ blocks, respectively, of the matrices
\begin{equation} \label{steadystateDPgains}
\begin{aligned}
   \bar{K}^{1} &= \left(\mathcal{Q}^{11} - \mathcal{Q}^{12} \left(\mathcal{Q}^{22}\right)^{-1} \mathcal{Q}^{12\top}\right)^{-1} \times \left(\mathcal{Q}^{12} \left(\mathcal{Q}^{22}\right)^{-1}  \mathcal{Q}^{02\top} -  \mathcal{Q}^{01\top} \right), \\
   \bar{K}^{2} &= \left(\mathcal{Q}^{22} - \mathcal{Q}^{12\top} \left(\mathcal{Q}^{11}\right)^{-1} \mathcal{Q}^{12} \right)^{-1} \times \left(\mathcal{Q}^{12\top} \left(\mathcal{Q}^{11}\right)^{-1}  \mathcal{Q}^{01\top} -  \mathcal{Q}^{02\top} \right), 
\end{aligned}
\end{equation}
from which we get $\bar{\mathbf{K}}^{1} = [K^{1}  \ \ -K^{1} \ \ 0_{m_1 \times n}] $ and $\bar{\mathbf{K}}^{2} = [K^{2}  \ \ 0_{m_2 \times n} \ \ -K^{2} ] $.

The constant term $r_t$ from the recursion in Proposition \ref{DPalgorithmLQ} does not converge, but if the quadratic component converges, then in steady state $r_t$ increases by a constant amount at each step that corresponds to the average steady-state cost
\begin{equation} \label{optimalaveragecost}
   J = \mathbf{tr}\left(P \mathbf{G} \mathbf{W} \mathbf{G}^T\right),
\end{equation}
with 
\begin{equation} \label{moresteadystateblockmatrices}
\begin{aligned}
\mathbf{G} &= \begin{bmatrix}
     I & 0 & 0 \\ I &  -L^1  & 0 \\ I & 0 & - L^2
    \end{bmatrix}, \quad
\mathbf{W} &= \begin{bmatrix}
     W & 0 & 0 \\ 0 &  V^1  & 0 \\ 0 & 0 & V^2
    \end{bmatrix}.
\end{aligned}
\end{equation}

We define the \textbf{one-step backward operator} as the right-hand side of \eqref{Bellman} using the feedback control gains \eqref{steadystateDPgains} for a given pair of filter parameters $F = ((A_1, \bar B_1, L_1), (A_2, \bar B_2, L_2))$ with common knowledge model $\mathcal{M}$, denoted by $\mathcal{B}(\mathcal{M}, P, F)$.

\subsection{Joint Convergence and Forward-Backward Value Iteration}
The steady-state equations \eqref{steadystatecovariance} and \eqref{Bellman} with \eqref{steadystateoptimalgains} and \eqref{steadystateDPgains} are coupled matrix equations for the steady-state covariance $\Sigma$, filter parameters $F$, cost matrix $P$, and feedback gains $K$. Using the one-step forward and backward operators defined in the previous subsections, we can express this as
\begin{equation} \label{jointsteadystate}
\begin{aligned}
   \Sigma &= \mathcal{F}\left(\mathcal{M}, \Sigma, K, \Gamma\right), \\
   P &= \mathcal{B}\left(\mathcal{M}, P, F\right),
\end{aligned}
\end{equation} 
These nonlinear matrix equations can be viewed as a set of coupled generalized
Riccati equations for estimation and control. If both the forward and backward
recursions simultaneously converge and the steady-state covariance $\Sigma$,
filter parameters $F$, cost matrix $P$, and feedback gains $K$ jointly satisfy
\eqref{jointsteadystate}, then we obtain equilibrium strategies for our dynamic
game. Once again, there are no existing algorithms to jointly solve these
coupled equations to the best of our knowledge. General numerical root finding
algorithms such as Newton's method could be applied and will be the subject of
future research. 

A natural approach, analogous to value iteration in single-player optimal control problems, is to iterate the right-hand side of \eqref{jointsteadystate} from an initial condition until convergence. We propose such a forward-backward value-iteration algorithm, described below in Algorithm \ref{algorithm:forward-backward-VI}.
\begin{algorithm}
\caption{Forward-Backward Value Iteration Algorithm}\label{algorithm:forward-backward-VI}
\begin{algorithmic}[1]
    \Require Common knowledge time-invariant model $\mathcal{M}$, inital error covariance $\Sigma$, initial gain $K$, cost $P$, and weight $\Gamma$ matrices, convergence tolerance $\epsilon$
    \While{$ \| \Sigma - \Sigma^- \| + \| P - P^- \| > \epsilon $}
        \State Store previous iterates: $\Sigma^- = \Sigma$, $P^- = P$ 
        \State One-Step Forward Operator: $\Sigma = \mathcal{F}\left(\mathcal{M}, \Sigma, K, \Gamma\right)$ \;
        \State Update filter parameters $F$ using \eqref{steadystateoptimalgains} and \eqref{steadystateparams} \;
        \State One-Step Backward Operator: $P = \mathcal{B}\left( \mathcal{M}, P, F \right) $ \;
        \State Update control gains $K$ using \eqref{steadystateDPgains} \; \;
    \EndWhile
    \Ensure Equilibrium steady-state gain matrices $K$, filter parameters $F$, cost matrix $P$, error covariance $\Sigma$
\end{algorithmic}
\end{algorithm}

Our implementation for this infinite-horizon algorithm is also available at \url{https://gitlab.com/scratch7473433/partialinfodynamicprogramming}. In the next section we present numerical experiments that demonstrate convergence of the algorithm to equilibrium strategies for our infinite-horizon stochastic dynamic game with asymmetric and partial information. 

\section{Numerical Experiments} \label{sec:numerical_experiments}
We illustrate our results in a two-player linear quadratic pursuit-evasion game with partial and asymmetric information. The player dynamics are
\begin{align}
    \begin{split}
        x_{t+1}^1 ={}& A x_t^1 + B^1 u_t^1 + G w_t^1, \\
        x_{t+1}^2 ={}& A x_t^2 + B^2 u_t^2 + G w_t^2,
    \end{split}
\end{align}
where $x_t^1$ is the state of the pursuer and $x_t^2$ is the state of the evader, and the disturbances $w_t^1$ and $w_t^2$ are zero mean with covariances $W_1$ and $W_2$, respectively. We consider time-discretized double integrator dynamics in the plane
\begin{equation*}
    A = \begin{bmatrix}
        1 & 0 & \Delta t & 0 \\
        0 & 1 & 0 & \Delta t \\
        0 & 0 & 1  & 0 \\
        0 & 0 & 0 & 1
    \end{bmatrix}, \quad B^1 = B^2 = \begin{bmatrix}
        0 & 0  \\
        0 & 0 \\
        \Delta t & 0 \\
        0 & \Delta t
    \end{bmatrix}, \quad W^1 = W^2 = 10^{-2} I_2, \quad G^1 = G^2 = \begin{bmatrix}
        0 & 0  \\
        0 & 0 \\
        1 & 0 \\
        0 & 1
    \end{bmatrix},
\end{equation*}
with $\Delta t = 0.1$. We define the relative state $x_t \coloneqq x_t^1 - x_t^2$, which evolves according to
\begin{equation}
    x_{t+1} = A x_t + B^1 u_t^1 - B^2 u_t^2 + w_t,
\end{equation}
where $w_t \sim \mathcal{N}(0, G(W^1 + W^2)G^\top)$. The cost parameters for the relative state are $\ Q = 10^{-3} I_4, \  R^1 = I_2, \  R^2 = -8 I_2$.

Both players obtain noisy relative position measurements
\begin{align}
    \begin{split}
        y_t^1 ={}& C^1 x_t + v_t^1, \\
        y_t^2 ={}& C^2 x_t + v_t^2,
    \end{split}
\end{align}
with
\begin{equation*}
    C^1 = C^2 = \begin{bmatrix}
        1 & 0 & 0 & 0 \\
        0 & 1 & 0 & 0
    \end{bmatrix},
\end{equation*}
where  $v_t^1$ and $v_t^2$ are independent  noise sequences with covariances $V_1$ and $V_2$, respectively, with $V_1 = V_2 = I_2$.

\subsection{Comparing strategies with limited vs. unlimited belief order}
We used the proposed method summarized in Algorithm \ref{algorithm:forward-backward-VI} to compute infinite-horizon equilibrium strategies for the pursuer and evader. Figure 1 shows the mean pursuer and evader trajectories for initial (relative) state and (relative) state estimates $$\left(x_0, z^1_0, z^2_0\right) = \left((-10, 0, 0, 0), (-10, 0, 0, -10), (-10, 0, 0, 0)\right).$$ 
This means that the pursuer and evader start at rest separated by $10m$ in the $x$ direction, and while the evader initially has exact estimates of the relative pursuer state, the pursuer has an incorrect initial estimate of the relative evader velocity in the $y$ direction.

When both players have perfect state information, the mean trajectories move along the line connecting the players' initial positions in this case, since the initial velocities are zero. However, with imperfect information the pursuer's incorrect initial estimate of the evader velocity causes an excursion in the $y$ direction shown in Figure 1. Pursuer and evader sample trajectories are shown in Figure 2, with random realizations of the measurement noise and disturbance sequences. The average infinite horizon cost in this instance computed with \eqref{optimalaveragecost} upon convergence of Algorithm \ref{algorithm:forward-backward-VI} is $2.872\times10^{-3}$.

For comparison, we also computed equilibrium strategies with higher-order beliefs of unlimited order using the best response method described in \cite{guan2025best}. In this method, in each step each player alternatively computes an optimal best response strategy with the opposing player's strategy held fixed. At each iteration, the internal state dimension of each player's
strategy increases by $n$, as the players use their output histories to estimate increasingly higher order beliefs. This process eventually converges to an equilibrium in this instance, when increasing the belief order does not provide any additional benefit. The average infinite horizon cost in this instance computed with this best response method is $2.895\times10^{-3}$. 

We further compared the proposed method (limited belief order) and the best response method (unlimited belief order) by evaluating their equilibrium values and associated strategies across different information settings. The results appear in Table \ref{tab1} and Figure \ref{fig:strategy_comparison}. 
When both players have comparable measurement uncertainties ($V^i \approx V^j$), strategies computed using limited belief order $n$ achieve equilibrium values within $1\%$ of those obtained with unlimited belief order. Moreover, players' mean trajectories under limited belief order strategies closely approximate those under unlimited belief order strategies. 
When measurement uncertainties differ significantly between players ($V^i \gg V^j$), the equilibrium value gap increases but remains bounded within $10\%$. However, the equilibrium strategies themselves may diverge substantially depending on the evader's input penalty. Specifically, a smaller (magnitude) $R^2$ produces a larger (magnitude) feedback gain $K^2$, which amplifies the difference between limited and unlimited belief order strategies.

These results demonstrate that the proposed method computes highly effective strategies for dynamic games with partial and asymmetric information in a computationally efficient manner, particularly when players have similar measurement uncertainties. When measurement uncertainties are asymmetric, the method's performance degrades with bounded error relative to the best response method.



\begin{figure}[htbp] 
\centering
\includegraphics[width=0.7\linewidth]{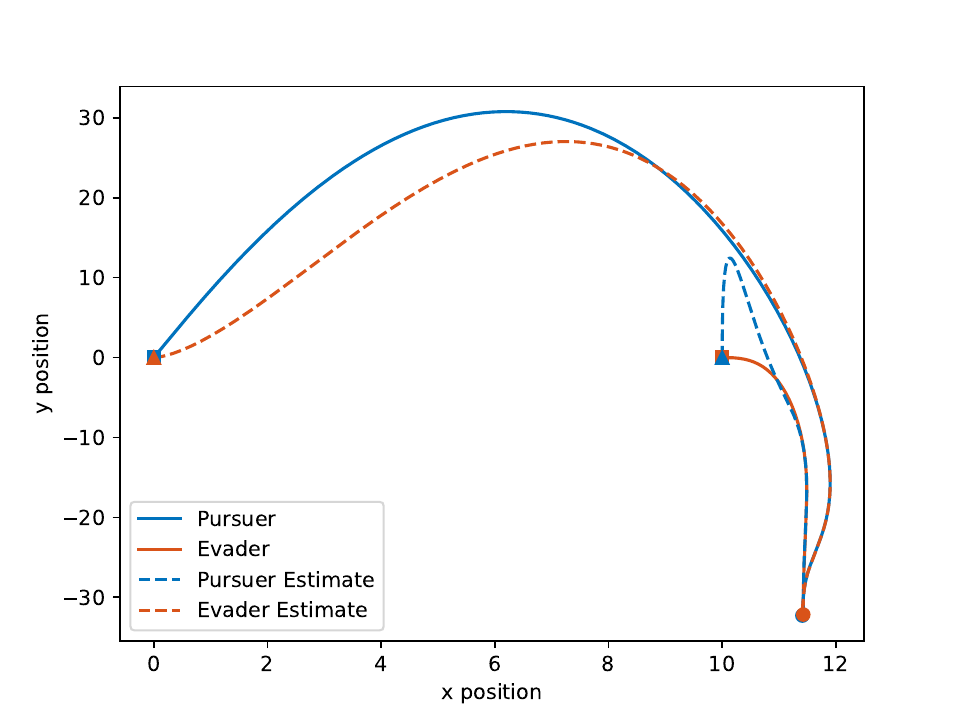}
\vspace{-0.5cm}
    \caption{Pursuer and evader mean trajectories (solid lines), along with their respective mean state estimates (dashed lines) from initial state $(x_0, z^1_0, z^2_0) = [(-10, 0, 0, 0), (-10, 0, 0, -10), (-10, 0, 0, 0)]$. The squares and triangles represent the initial positions and initial position estimates of each player, respectively. The circles represent the final positions. Total Cost: $2.872\times10^{-3}$.}
    \label{fig1}
\end{figure}

\begin{figure}[htbp] 
\centering
\includegraphics[width=0.7\linewidth]{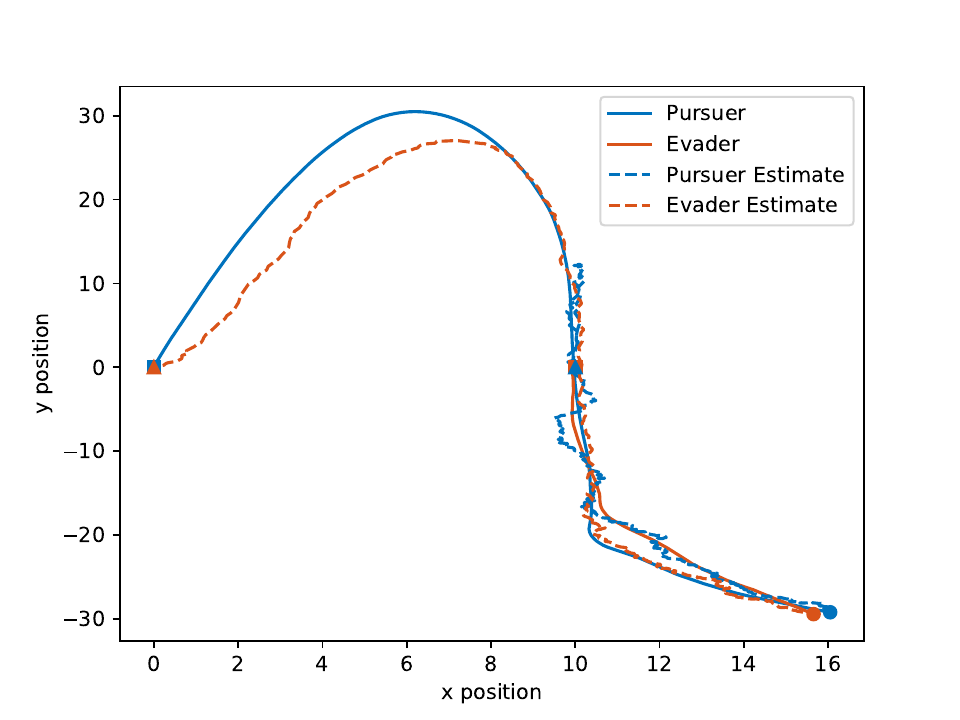}
\vspace{-0.5cm}
    \caption{Pursuer and evader sample trajectories (solid lines), along with their respective sample estimates (dashed lines) from initial state $(x_0, z^1_0, z^2_0) = [(-10, 0, 0, 0), (-10, 0, 0, -10), (-10, 0, 0, 0)]$. The squares and triangles represent the initial positions and initial position estimates of each player, respectively. The circles represent the final positions.}
    \label{fig1}
\end{figure}

\begin{center}
\begin{table*}[htbp]%
\caption{Performance gap under limited vs. unlimited belief order.\label{tab1}}
\begin{tabular*}{\textwidth}{@{\extracolsep\fill}lllll@{}}
\toprule
\textbf{Cases} & \textbf{LQG cost}  & \textbf{Total cost $J_{ubo}$$^{\tnote{\bf a}}$} & \textbf{Total cost $J_{lbo}$$^{\tnote{\bf b}}$}  & \begin{tabular}{@{}l@{}}\textbf{Performance gap ratio} \\ $\|J_{ubo} - J_{lbo} \|/J_{ubo}$ \end{tabular}  \\
\midrule
$V^1=V^2=I_2$, $R^2=-8I_2$ & $2.625\times 10^{-3}$ & $2.895 \times 10^{-3}$  & $2.872 \times 10^{-3}$  & $1\%$ \\
$V^1=10^{-6}I_2$, $V^2=10^6I_2$, $R^2=-2.6I_2$ & $1.092 \times 10^{-3}$  & $1.092 \times 10^{-3}$  & $1.193 \times 10^{-3}$  & $8\%$ \\
$V^1=10^{6}I_2$, $V^2=10^{-6}I_2$, $R^2=-1.3\times 10^5I_2$ & $3.8583$  & $4.0241$  & $3.8623$ & $4\%$ \\
\bottomrule
\end{tabular*}
\begin{tablenotes}
\item[$^{\rm a}$] Game's total cost with unlimited belief order.
\item[$^{\rm b}$] Game's total cost with limited belief order.
\end{tablenotes}
\end{table*}
\end{center}


\begin{figure*}[htbp]
    \centering
    \begin{subfigure}[b]{0.32\textwidth}
        \includegraphics[width=\textwidth]{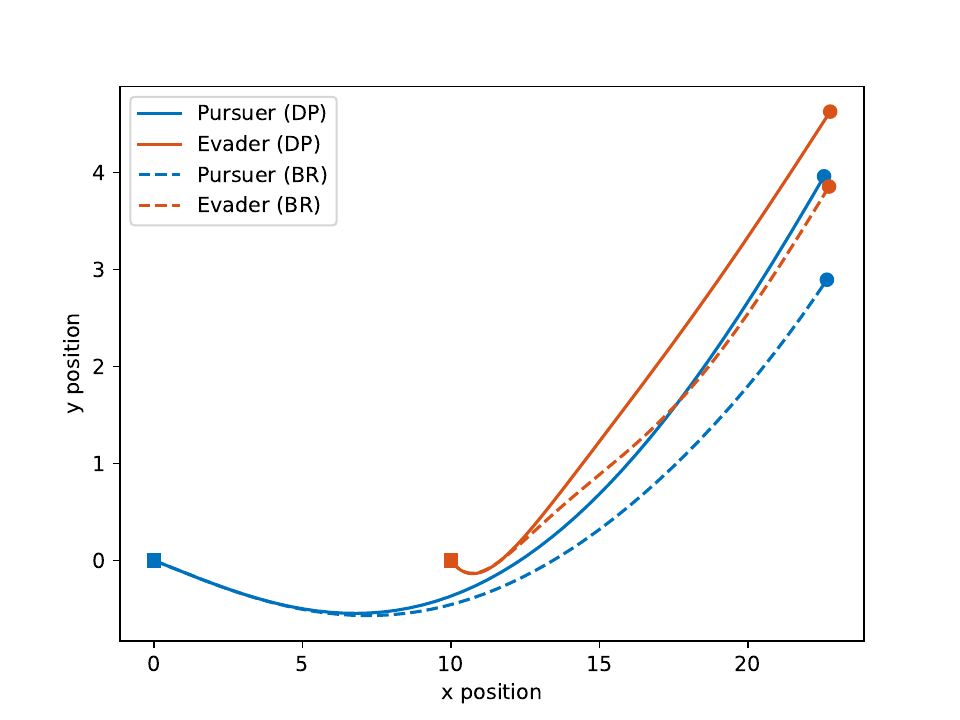}
        \caption{$V^1=V^2=I_2$, $R^2=-8I_2$.}
    \end{subfigure}
    \hfill
    \begin{subfigure}[b]{0.32\textwidth}
        \includegraphics[width=\textwidth]{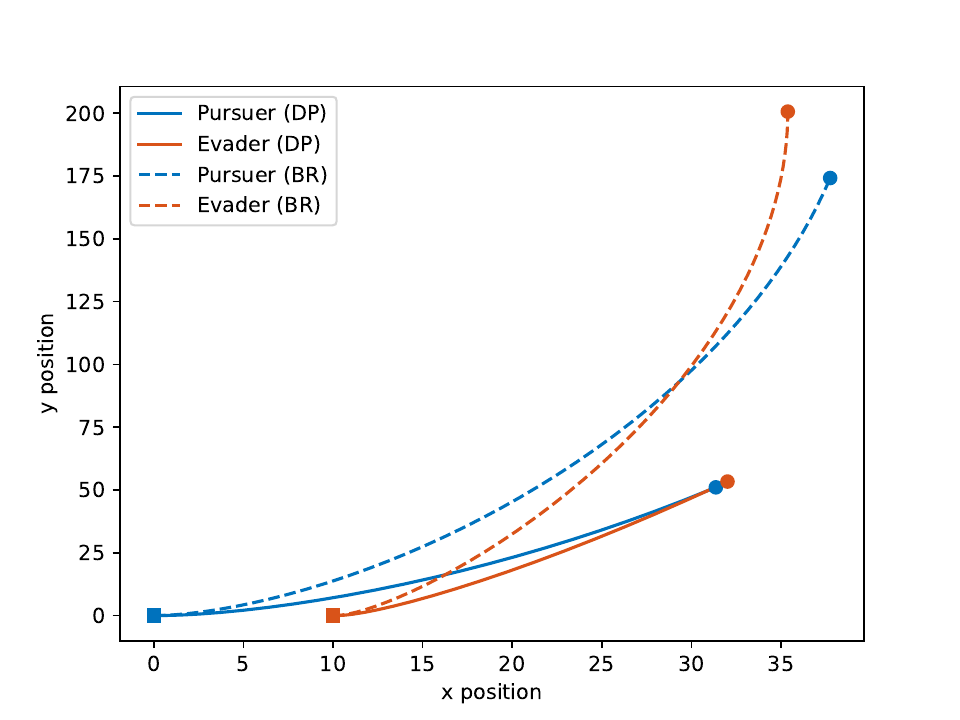}
        \caption{$V^1=10^{-6}I_2$, $V^2=10^6I_2$, $R^2=-2.6I_2$.}
    \end{subfigure}
    \hfill
    \begin{subfigure}{0.32\textwidth}
        \includegraphics[width=\textwidth]{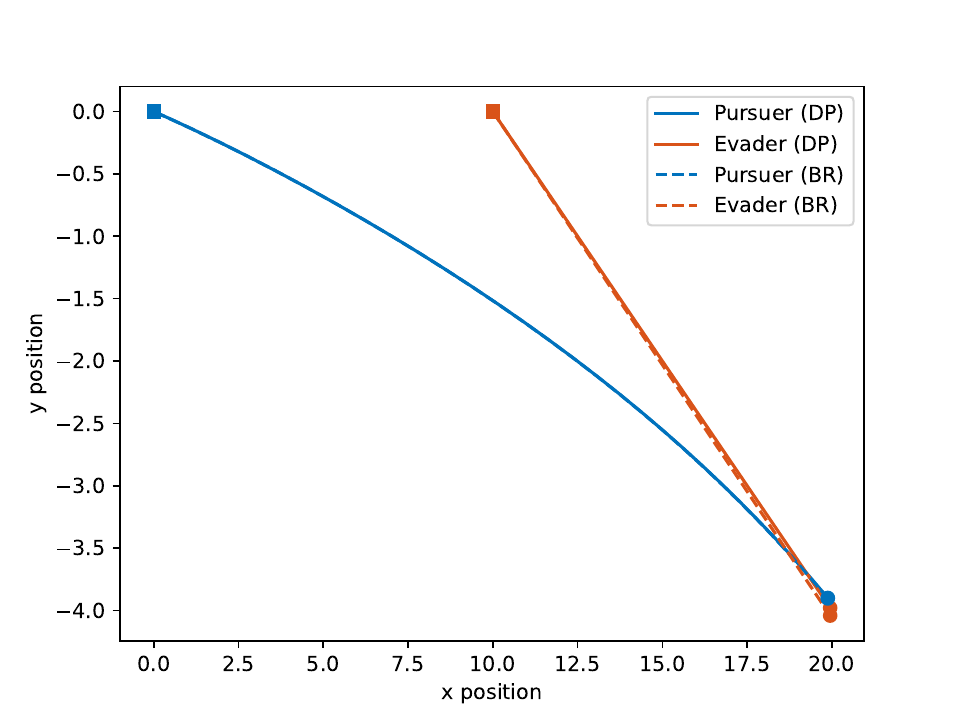}
        \caption{$V^1=10^{6}I_2$, $V^2=10^{-6}I_2$, $R^2=-1.3\times10^5I_2$.}
    \end{subfigure}
    \caption{Comparison of equilibrium strategies under limited vs. unlimited belief orders, starting from initial state $(x_0, z^1_0, z^2_0) = [(-10, 0, -0.5, 0.2), (-10, 0, -0.5, 0.2), (-10, 0, -0.5, -4.8)]$. DP: strategies computed using forward-backward dynamic programming with limited belief order. BR: strategies computed using best response dynamics with unlimited belief order.}
    \label{fig:strategy_comparison}
\end{figure*}

\subsection{Exploring asymmetries in controllability and observability}
Since the advent of differential game theory in the seminal work of Isaacs \cite{isaacs1965differential}, problems with asymmetries in controllability/maneuverability have served as important benchmarks. A classic example is the homicidal chauffeur problem, a pursuit-evasion game which pits a slow but highly maneuverable pedestrian against a faster but less maneuverable vehicle. Also, a key result in linear quadratic pursuit evasion games reveals that successful capture of an evader depends on their relative controllability as quantified by certain relative controllability Gramians \cite{ho1965differential}.

Dynamic games with partial and asymmetric information also brings the relative observability properties of the players into the picture. It is then the combination of controllability and observability properties of the players that determines the value and outcome of the game. For instance, Figure 3 shows the pursuer and evader mean trajectories when the pursuer is less maneuverable ($B^1_{24} = 0.7 \Delta t$, all other parameters the same as above), with the strategies again computed with Algorithm \ref{algorithm:forward-backward-VI}. The average infinite horizon cost in this instance is $3.833\times10^{-3}$, indicating a decrease in the pursuer's performance due to the reduced maneuverability. Similarly, Figure 4 shows the pursuer and evader mean trajectories when the pursuer has noisier sensor measurements in the $y$ direction ($V^1 = \mathrm{diag}(1, 50)$), which reduces its observability of the evader's relative state. The average infinite horizon cost in this instance is $7.572\times10^{-3}$, again indicating a decrease in the pursuer's performance due to the reduced observability.

It is entirely possible for the pursuer to be \emph{more} maneuverable but have decreased observability capabilities outweigh its maneuverability advantages. Figure 5 shows the pursuer and evader mean trajectories when the pursuer is more maneuverable ($B^1_{24} = 1.5 \Delta t$) but has noisier sensor measurements in the $y$ direction ($V^1 = \mathrm{diag}(1, 50)$), with the strategies again computed with Algorithm \ref{algorithm:forward-backward-VI}. The average infinite horizon cost in this instance is $5.539\times10^{-3}$, indicating that the pursuer's manueverabiliity advantage in the $y$ direction is overwhelmed by its noisier sensor measurements in the $y$ direction. It's even possible for the upper value of the game to become unbounded even when the pursuer has a maneuverability advantage, if its observability is sufficiently compromised.

\begin{figure}[htbp] 
\centering
\includegraphics[width=0.7\linewidth]{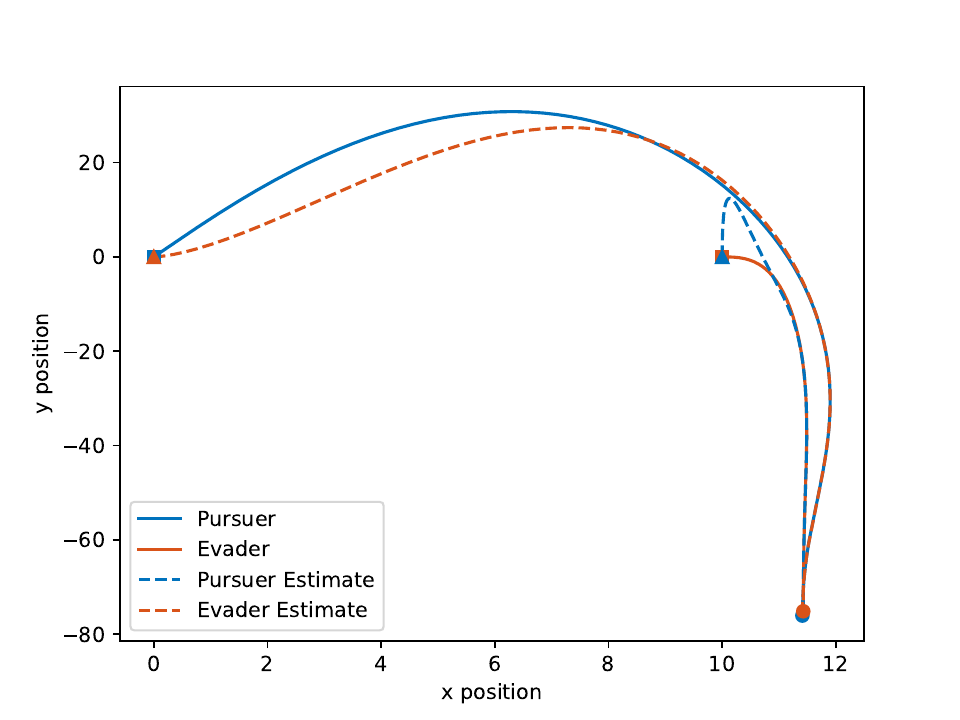}
\vspace{-0.5cm}
    \caption{\textbf{Less maneuverable pursuer.} ($B^1_{24} = 0.7 \Delta t$) Pursuer and evader mean trajectories, along with their respective mean sample estimates from a prescribed initial state. Total Cost: $3.833\times10^{-3}$.}
    \label{fig1}
\end{figure}

\begin{figure}[htbp] 
\centering
\includegraphics[width=0.7\linewidth]{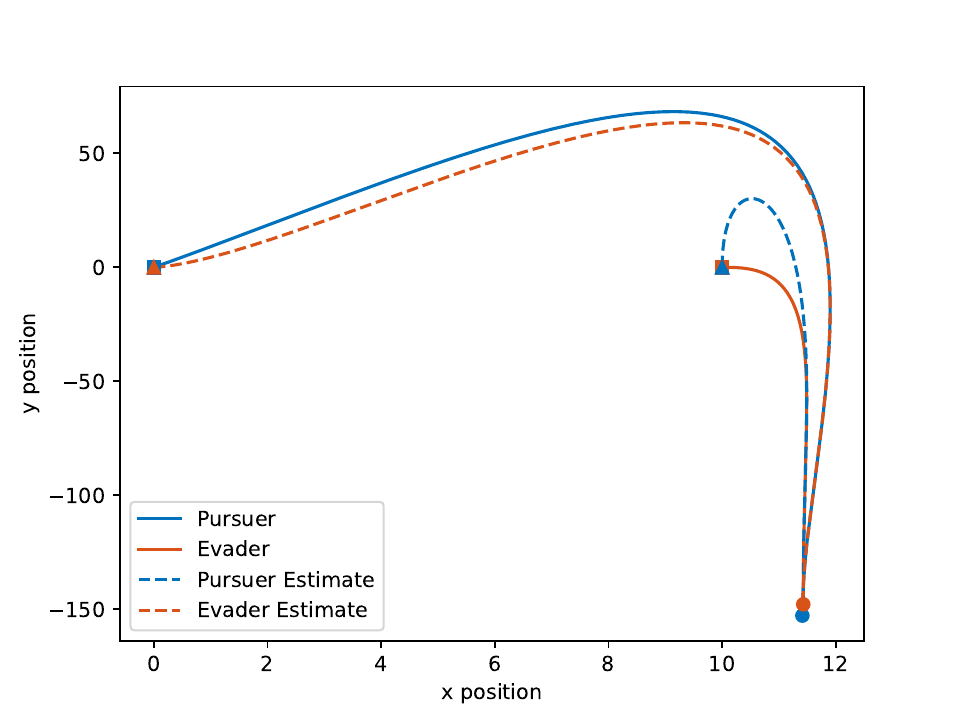}
\vspace{-0.5cm}
    \caption{\textbf{Pursuer has noisier observations.} ($V^1 = \mathrm{diag}(1, 50)$) Pursuer and evader mean trajectories, along with their respective mean sample estimates from a prescribed initial state. Total Cost: $7.572\times10^{-3}$.}
    \label{fig1}
\end{figure}

\begin{figure}[htbp] 
\centering
\includegraphics[width=0.7\linewidth]{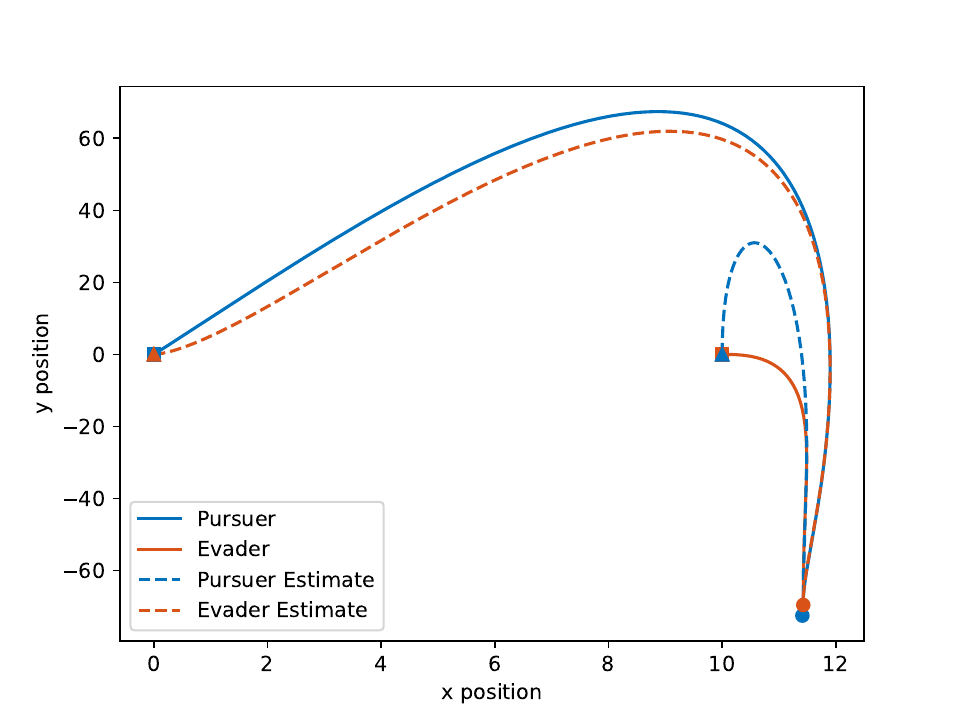}
\vspace{-0.5cm}
    \caption{\textbf{More maneuverable pursuer but with noisier observations.} ($B^1_{24} = 1.5 \Delta t$, $V^1 = \mathrm{diag}(1, 50)$) Pursuer and evader mean trajectories, along with their respective mean sample estimates from a prescribed initial state. Total Cost: $5.539\times10^{-3}$.}
    \label{fig1}
\end{figure}

\section{Extensions, Variations, and Open Problems} \label{sec:extensions}
There are many possible extensions and variations of the proposed modeling framework and methodology. These are being pursued in ongoing research.

\textbf{Modeling variations within the linear quadratic setting.} Several variations of the modeling framework that remain within the general linear quadratic setting are straightforward to handle with minor modifications to the proposed methodology. These include trajectory tracking problems, affine dynamics (or equivalently, non-zero mean disturbances), and general quadratic costs with terms that couple the state and inputs. 

\textbf{Multiplicative noise.} A more powerful generalization that can also be handled within the proposed frameworks is dynamics and measurement models with multiplicative noise (rather than just additive noise). In multiplicative noise models, the system matrices $A_t, B_t^1, B_t^2, C_t^1, C_t^2$ are allowed to be random matrices (rather than fixed) with a prescribed covariance structure over their entries. This framework enables parametric uncertainty in the system model to be encoded into the multiplicative noise to promote robustness to these uncertainties in the feedback strategies. Optimal linear output dynamic feedback strategies for these models can be computed with dynamic programming and robust Kalman filtering techniques \cite{wonham1967optimal,gravell2020learning,gravell2022policy}.

\textbf{Non-zero sum $N$-player games.} Another important extension is the non-zero sum setting with more than two players. Each player has its own cost function that in general is not directly aligned with those of other players and may feature both cooperative and non-cooperative behavior. Equilibrium concepts, including Nash equilibria and refinements, become richer and more varied. Coupled dynamic programming recursions can still be used for the backward pass in nonzero sum games \cite{bacsar1998dynamic}.

\textbf{Nonlinear dynamics and measurement models.}
Virtually all application areas feature nonlinear dynamics and measurement models. The solution for the linear quadratic dynamic game can be used as an approximation for solutions to nonlinear dynamic games. The most straightforward approximation is to simply linearize the nonlinear system about an equilibrium or a nominal trajectory, and use the linearized system matrices in the algorithms. A more sophisticated approach involves iterative linearization techniques along the lines of iLQR and differential dynamic programming. Adapting the proposed methodology to nonlinear models using iterative linearization techniques will significantly expanding the applicability and impact.

\textbf{Semidefinite programming formulations.} Many optimal control and dynamic game design problems can be formulated using semidefinite programming \cite{boyd1994linear,gahinet1994linear}. It may be possible to obtain analogous semidefinite programming formulations of the solutions proposed here.

\textbf{Forward-backward operator properties.}
The properties of the forward-backward operator described in Section \ref{Section:InfHorizon} may reveal interesting insights about convergence properties of the proposed algorithms. For example, convergence of value iteration in single-player optimal control is directly related to contractivity of the Bellman operator.

\section{Conclusions}\label{sec:conclusion}
We have formulated and studied a class of two-player zero-sum linear-quadratic stochastic dynamic games with partial and asymmetric information. To avoid an infinite regress of higher-order beliefs amongst agents and obtain computationally implementable results, we considered strategies with limited internal state dimension. We presented a novel iterative forward-backward algorithm to jointly compute belief states and equilibrium strategies and value functions for a finite-horizon problem. and a value iteration-like algorithm to jointly compute stationary belief states and equilibrium strategies for an average-cost infinite-horizon problem. An open-source implementation of the algorithms is available at \url{https://gitlab.com/scratch7473433/partialinfodynamicprogramming}.



\bmsection*{Acknowledgments}
This work was supported by the United States Air Force Office of Scientific Research under Grants FA9550-23-1-0424 and FA2386-24-1-4014, and by the National Science Foundation under Grant ECCS-2047040.





\bibliography{wileyNJD-AMA}

\end{document}